\documentclass[letterpaper,11pt]{amsart}
\usepackage{amsmath,amssymb,amsthm,pinlabel,tikz,hyperref,mathrsfs,color, thmtools}
\usepackage{verbatim}
\usepackage{float}
\usepackage{caption}
\usepackage{subcaption}
\newcommand{\nc}{\newcommand}
\nc{\dmo}{\DeclareMathOperator}
\dmo{\ra}{\rightarrow}
\dmo{\Prob}{\mathbb{P}}
\dmo{\E}{\mathbb{E}}
\dmo{\N}{\mathbb{N}}
\dmo{\Z}{\mathbb{Z}}
\dmo{\Q}{\mathbb{Q}}
\dmo{\R}{\mathbb{R}}
\dmo{\C}{\mathcal{C}}
\dmo{\X}{\mathcal{X}}
\dmo{\U}{\mathcal{U}}
\dmo{\T}{\mathcal{T}}
\dmo{\F}{\mathcal{F}}
\dmo{\AC}{\mathcal{AC}}
\dmo{\w}{\omega}
\dmo{\Mod}{Mod}
\dmo{\PMod}{PMod}
\dmo{\PMF}{\mathcal{PMF}}
\dmo{\Mat}{Mat}
\dmo{\supp}{supp}
\dmo{\vol}{vol}
\dmo{\B}{B}
\dmo{\PB}{PB}
\dmo{\PR}{PSL(2,\mathbb{R})}
\dmo{\GL}{GL(k, \mathbb{C})}
\dmo{\SL}{SL(2, \mathbb{Z})}
\dmo{\RP}{\mathbb{R} \mathrm{P}}
\dmo{\I}{\mathcal{I}}
\dmo{\el}{\ell_{\C}}
\dmo{\NN}{\mathcal{N}}
\dmo{\rk}{rank}
\dmo{\tr}{tr}
\dmo{\llangle}{\langle\langle}
\dmo{\rrangle}{\rangle\rangle}
\dmo{\Unif}{Unif}

\usetikzlibrary{decorations.markings}
\tikzset{->-/.style={decoration={
  markings,
  mark=at position #1 with {\arrow{>}}},postaction={decorate}}}

\nc{\nt}{\newtheorem}

\nt{theorem}{Theorem}

\newtheorem{thm}{{\bf Theorem}}[section]
\newtheorem{lem}[thm]{{\bf Lemma}}
\newtheorem{cor}[thm]{{\bf Corollary}}
\newtheorem{prop}[thm]{{\bf Proposition}}

\newtheorem{remark}[thm]{Remark}

\newtheorem{definition}[thm]{Definition}
\numberwithin{equation}{section}

\title[Topological entropy from a typical Thurston's construction]{Topological entropy of pseudo-Anosov maps from a typical Thurston's construction}

\date{\today}
\author{Hyungryul Baik}
\address{%
		Department of Mathematical Sciences, KAIST\\
		291 Daehak-ro Yuseong-gu, Daejeon, 34141, South Korea 
}
\email{%
        hrbaik@kaist.ac.kr
}

\author{Inhyeok Choi}

\address{%
		Department of Mathematical Sciences, KAIST\\
		291 Daehak-ro Yuseong-gu, Daejeon, 34141, South Korea 
}
\email{%
        inhyeokchoi@kaist.ac.kr
        }

\author{Dongryul M. Kim}

\address{%
		Department of Mathematical Sciences, KAIST\\
		291 Daehak-ro Yuseong-gu, Daejeon, 34141, South Korea 
}
\email{%
        dongryul.kim@kaist.ac.kr
}

\begin{document}
\begin{abstract}
	
In this paper, we develop a way to extract information about a random walk associated with a typical Thurston's construction. We first observe that a typical Thurston's construction entails a free group of rank 2. We also present a proof of the spectral theorem for random walks associated with Thurston's construction that have finite second moment with respect to the Teichm\"uller metric. Its general case was remarked by Dahmani and Horbez. Finally, under a hypothesis not involving moment conditions, we prove that random walks eventually become pseudo-Anosov.
	
As an application, we first discuss a random analogy of Kojima and McShane's estimation of the hyperbolic volume of a mapping torus with pseudo-Anosov monodromy. As another application, we discuss non-probabilistic estimations of stretch factors from Thurston's construction and the powers for Salem numbers to become the stretch factors of pseudo-Anosovs from Thurston's construction.
	


\noindent{\bf Keywords.} Topological entropy, Thurston's construction, pseudo-Anosov surface homeomorphism, Markov chains, hyperbolic volume of mapping torus, Salem numbers.

\noindent{\bf MSC classes:} 37B40, 37D40, 37E30, 60B15, 60G50.
\end{abstract}

\maketitle

%
%

\section{Introduction}	\label{sec:introduction}

For a closed orientable connected surface $S$ (or $S_g$ when $g$ plays a role) of genus $g > 1$, its \emph{mapping class group} $\Mod(S)$ is defined as follows. $$\Mod(S) := \mathrm{Homeo}^{+}(S) / \mbox{Isotopy}$$ The elements of $\Mod(S)$ are called \textit{mapping classes} and classified, by the Nielsen--Thurston classification, into three categories: periodic, reducible, and pseudo-Anosov elements. Among them, pseudo-Anosov mapping classes are considered the generic ones (cf. \cite{erlandsson2020genericity}) but it is not straightforward to come up with concrete examples. 

In this regard, some constructions for pseudo-Anosovs have been developed, such as Thurston's construction \cite{thurston1988geometry} and Penner's construction \cite{penner1988construction}. Both constructions use Dehn twists along a pair of multicurves which fills the surface. Penner originally conjectured that every pseudo-Anosov mapping class has some power arising from Penner's construction. However, Shin and Strenner showed in \cite{shin2016pseudo} that coronal pseudo-Anosov mapping classes, which are characterized by their stretch factors, have no powers constructed from Penner's construction. This answers Penner's conjecture negatively.

We can then ask whether a typical pseudo-Anosov mapping class is obtained from Penner's construction. We can also ask an analogous question for Thurston's construction. Motivated by this, we study the behavior of pseudo-Anosov maps obtained from Thurston's construction, especially in a typical circumstance.

Let $(A, B)$ be a filling pair of multicurves on $S$. Thurston's construction describes the subgroup $\langle T_A, T_B \rangle$ of $\Mod(S)$ generated by multitwists $T_A$ and $ T_B$ along $A$ and $B$, respectively. In \cite{leininger2004groups}, Leininger provided a necessary and sufficient condition for $(A, B)$ to result in $\langle T_A, T_B \rangle \cong F_2$, the free group of rank 2. It asserts that $\langle T_A, T_B \rangle \ncong F_2$ if and only if $A$ and $B$ have a certain tree-like position as in Figure \ref{fig:treelike}. We summarize this as follows:

\begin{restatable*}{obs}{typicalthurston}\label{typicalthurston}

For a typical filling pair of multicurves $(A, B)$, we have $\langle T_A, T_B \rangle \cong~F_2$.
	
\end{restatable*}

From now on, we assume the conclusion of Observation \ref{typicalthurston}. We then consider probability measures supported on the subgroup $\langle T_{A}, T_{B} \rangle$ and random mapping classes arising from them. 

Let $\nu : \Mod(S) \to [0, 1]$ be a probability measure with the support $\supp \nu := \{g \in \Mod(S) : \nu(g) > 0 \}$ contained in $\langle T_A, T_B \rangle$. Let us also fix an initial mapping class $g_{0} \in \Mod(S)$, which is assumed to be the identity element unless otherwise stated. Then $\nu$ induces a probability measure $\Prob$ on the space of sample paths $\Mod(S)^{\N}$ as follows. Given elements $g_{1}, \ldots, g_{n} \in \Mod(S)$, we define a cylinder set \[
[g_1, \ldots, g_n] := \{ \w = (\w_i) \in \Mod(S)^{\N} : \w_i = g_i \mbox{ for } 1 \le i \le n \}
\]
and assign \[
\Prob([g_1, \ldots, g_n]) = \nu(g_{0}^{-1}g_{1}) \nu(g_{1}^{-1}g_{2}) \cdots \nu( g_{n-1}^{-1}g_{n}).
\]
Throughout the paper, $\w$ stands for a random walk or a sample path on $\Mod(S)$, and $\w_n$ for the $n$-th step mapping class of $\w$.

In this setting, we investigate how much and how often $\w$ produces pseudo-Anosov maps. This topic was handled by Dahmani and Horbez in a more general setting, and we present one of their spectral theorems below. In the following, $\nu$ is said to be \text{non-elementary} if $\langle \supp \nu \rangle$ contains a pair of pseudo-Anosov maps with disjoint sets of fixed points in $\PMF(S)$.

\begin{restatable*}[Dahmani--Horbez, \cite{dahmani2018spectral}]{thm}{DH} \label{thm:DH}
	Let $\nu$ be a non-elementary probability measure on $\Mod(S)$ with finite support, and $\Prob$ be the induced measure on $\Mod(S)^{\N}$. Then for $\Prob-a.e.$ $\w\in \Mod(S)^{\N}$, $\w_n$ is eventually pseudo-Anosov. Moreover,  the stretch factor $\lambda_{\w_n}$ of a pseudo-Anosov $\w_n$ satisfies $$\lim_{n \to \infty} {1 \over n} \log \lambda_{\w_n} = L_{\T} > 0$$ where the constant $L_{\T} := \lim_{n \to \infty} {1 \over n} d_{\T}(\X, \w_n \cdot \X)$ is the \emph{drift} of $\w$ with respect to the Teichm{\"u}ller metric $d_{\T}$.
\end{restatable*}

According to Masai \cite{masai2018topological}, the drift $L_{\T}$ is almost surely equal to the topological entropy $h(\w)$ of $\w$ if $\nu$ is non-elementary and has finite first moment with respect to the Teichm{\"u}ller metric. See Theorem \ref{thm:masai}. In addition, the topological entropy $h(\varphi)$ of a pseudo-Anosov map $\varphi$ is equal to $\log \lambda_{\varphi}$, where $\lambda_{\varphi}$ is the stretch factor of $\varphi$ (cf. \cite[Expos\'e Ten]{FLP}). Thus, Dahmani and Horbez's result can be interpreted as an asymptotic relation between the topological entropy of the entire random walk $\w$ and that of each random mapping class $\w_n$. To keep this interpretation, we write spectral theorems in terms of $h(\w)$ even though it does not depend on the choice of sample path $\w$ by \cite{masai2018topological}.

In \cite[Remark 2.7, 3.2]{dahmani2018spectral}, the authors assert that the finite support condition in Theorem \ref{thm:DH} can be replaced with the finite second moment condition. For the sake of completeness, we present the proof of this assertion in our setting.

\begin{restatable}{theorem}{secondmoment}
Let $(A, B)$ be a filling pair of multicurves on $S$ and $\nu$ be a non-elementary probability measure on $\langle T_A, T_B \rangle$. Suppose that $\nu$ has finite second moment with respect to the Teichm\"uller metric, i.e., for some $\X \in \T(S)$, $$\sum_{g \in \Mod(S)} d_{\T}(\X, g \cdot \X)^2 \nu(g) < \infty.$$ Then for $\Prob-a.e.$ $\w \in \Mod(S)^{\N}$, $\w_{n}$ is eventually pseudo-Anosov and  $$\lim_{n \to \infty} {1 \over n} \log \lambda_{\w_n} = h(\w).$$
\end{restatable}

As noted in this theorem, random mapping classes eventually become pseudo-Ansov \emph{almost surely} if the second moment is finite. In contrast, without moment condition, previously known results are in the sense of \emph{phenomena in probability}, which is weaker than almost sure phenomena. For example, Rivin \cite{rivin2008walks} and Kowalski \cite{Kowalski2008sieve} proved that nearest-neighbor random walks on $\Mod(S)$ become pseudo-Anosov with asymptotic probability one. Note that their strategies are applicable in a broader setting, not only for mapping class groups. In addition, Maher proved in \cite{maher2011random} the convergence $\Prob(\w_n \mbox{ is pseudo-Anosov}) \to 1$ as $n \to \infty$ if the governing probability measure $\nu$ on $\Mod(S)$ is non-elementary. 

Hence, it is natural to ask whether an almost sure phenomenon can be obtained without any moment condition. We answer this question in the setting of Thurston's construction, utilizing facts about free groups.

\begin{restatable}{theorem}{eventuallypA} \label{theorem:C}
	Let $(A, B)$ be a filling pair of multicurves on $S$ such that $\langle T_A, T_B \rangle \cong F_{2}$, and $\nu : \Mod(S) \to [0, 1]$ be a non-elementary probability measure with $\supp \nu \subseteq \langle T_A, T_B \rangle$ and $\rk \langle \supp \nu \rangle <  \infty$. Then for $\Prob-a.e.$ $\w \in \Mod(S)^{\N}$, there exists $N \in \N$ such that $$\w_n \mbox{ is pseudo-Anosov for all } n > N.$$
\end{restatable}

There are reasons why we focus on Thurston's construction. First, Thurston's construction makes use of a Teichm{\"u}ller disk $f(\mathbb{H}^{2})$ in $\T(S)$, an isometrically embedded copy of the hyperbolic disk $\mathbb{H}^{2}$ of curvature $-4$. Mapping classes from Thurston's construction fix $f(\mathbb{H}^{2})$ and act as isometries of $f(\mathbb{H}^{2})$ (\cite{herrlich2007boundary}, \cite{thurston1988geometry}, \cite{leininger2004groups}). Hence, the Teichm\"uller distance can be translated into the hyperbolic distance, allowing us an explicit computation.

Moreover, let $\rho : \langle T_A, T_B \rangle \rightarrow \PR$ be the representation obtained by Thurston's construction. Then there exists a point $\mathcal{X} \in f(\mathbb{H}^{2})$ such that \[
\log \| \rho(w)\| = d_{\T}(\mathcal{X}, w \cdot \mathcal{X})
\]
for each $w \in \langle T_A, T_B \rangle$, where $\|T\| := \sup_{|x|=1}|Tx|$ is the operator norm. If a measure $\nu$ has finite first moment with respect to the Teichm\"uller metric, then $$\int \log \lVert \rho(\w_1) \rVert d\nu < \infty.$$ 
In this case, using the following theorem of Furstenberg and Kesten, we can approximate the topological entropy $h(\w)$ without referring to individual sample paths. We remark that an explicit formula for $h(\w)$ can be obtained by the work \cite{pollicott2010maximal} of Pollicott when $\nu$ is supported on finitely many mapping classes that correspond to strictly positive matrices in $\mathrm{SL}(2, \R)$.

\begin{thm}[Furstenberg--Kesten, \cite{furstenberg1960products}] \label{thm:furstenberg}
	Let $(X_n)$ be an independent and identically distributed stochastic process with values in the set of $k \times k$ matrices such that $$\E \max \{\log \lVert X_1 \rVert, 0\} < \infty.$$ Then as $n \to \infty$, $${1 \over n} \log \lVert X_1 \cdots X_{n-1} X_n \rVert \to E \quad \mbox{almost surely},$$ where $E = \lim_{n \to \infty} {1 \over n} \E \log \lVert X_1 \cdots X_{n-1} X_n \rVert $.
\end{thm}

Especially, the subadditivity of $\E \log \lVert X_1 \cdots X_{n-1} X_n\rVert$ implies that $E = \lim_{n \to \infty} {1 \over n} \E \log \lVert X_1 \cdots X_{n-1} X_n \rVert =\inf_{n} {1 \over n} \E \log \lVert X_1 \cdots X_{n-1} X_n \rVert $.
This allows us to approximate or provide an explicit upper bound for $h(\w)$.

As an application of above results, we estimate the hyperbolic volumes of random mapping tori arising from a typical Thurston's construction. In \cite{kojima2018normalized}, Kojima and McShane fixed a pseudo-Anosov map $\varphi$ and investigated the hyperbolic volume of a mapping torus with monodromy $\varphi$. It represents the dynamics of the iterated sequence $\varphi, \varphi^2, \ldots, \varphi^n, \ldots$.

In order to translate this into a random version, we replace the iterate $\varphi^{n}$ with the element $\w_{n}$ of a random walk associated with a typical Thurston's construction. Thanks to Theorem \ref{theorem:C}, such elements are eventually pseudo-Anosov almost surely, and thus one can discuss the hyperbolic volume for the mapping torus with monodromy $\w_n$ for large $n$. Consequently, our random analogy of Kojima and McShane's estimate follows:

\begin{restatable}{theorem}{mappingtori}
	Let $(A, B)$ be a filling pair of multicurves on $S$ such that $\langle T_A, T_B \rangle \cong F_{2}$ and $\nu : \Mod(S) \to [0, 1]$ be a non-elementary probability measure with $\supp \nu \subseteq \langle T_A, T_B \rangle$ and $\rk \langle \supp \nu \rangle <  \infty$. Suppose also that $\nu$ has finite first moment with respect to the Teichm\"uller metric. Then $$\limsup_{n \to \infty} {1 \over n} \vol(M_{\w_n}) \le -3\pi\chi(S) \cdot h(\w)$$ for $\Prob-a.e.$ $\w \in \Mod(S)^{\N}$ where $\chi(S)$ is the Euler characteristic of $S$.
\end{restatable}

This estimation supports the point of view that the topological entropy of a random walk shares common characteristics with the entropy of a single pseudo-Anosov. Indeed, Kojima and McShane's estimate asserts that $${1 \over n} \vol(M_{\varphi^n}) \le -3\pi \chi(S) \cdot h(\varphi)$$ where $h(\varphi)$ is the topological entropy of a  pseudo-Anosov $\varphi$, which can be regarded as an interchange of $h(\w)$ and $h(\varphi)$. See Remark \ref{rem:kojimamcshane} for details.

As another application, we establish an estimation for stretch factors obtained by Thurston's construction; note that the estimation is non-probabilistic.

\begin{restatable}{theorem}{distribution} \label{theorem:E}
	Let $(A, B)$ be a filling pair of multicurves on $S$. For a finite generating set $\Lambda$ of $\langle T_A, T_B \rangle$, let $| \cdot |_{\Lambda}$ be a word norm with respect to $\Lambda$. Then there exist constants $L > 0$ and $N$ such that $$\log \lambda_{\varphi} < L |\varphi|_{\Lambda}$$ for all pseudo-Anosov $\varphi \in \langle T_A, T_B \rangle$ with $|\varphi|_{\Lambda} > N$.

\end{restatable}

See Subsection \ref{subsec:distribution} for the detailed nature of $L$. In particular, for some choice of $\Lambda$, $L$ can be much smaller than the first moments of measures on $\Lambda$. In contrast, $L$ is always greater than the drift $L_{\T}$ of a measure on $\Lambda$. 


As stretch factors of pseudo-Anosovs from Thurston's construction are related to Salem number, an algebraic unit $\lambda$ whose all Galois conjugates except $\lambda^{\pm 1}$ lie on the unit circle in $\mathbb{C}$, one can also obtain a consequence regarding Salem numbers. Pankau proved in \cite{pankau2017salem} that every Salem number has a power, which is the stretch factor of a pseudo-Anosov arising from Thurston's construction. Regarding this, one can ask how large the power should be. We provide one estimation here.

\begin{restatable}{theorem}{salempower}
	Let $\lambda$ be a Salem number and suppose that $\lambda^k$ is the stretch factor of a pseudo-Anosov $\varphi$ obtained from Thurston's construction with $\rho: \langle T_A, T_B \rangle \to \PR$. Then$$ k \le {K |\varphi| \over \log \lambda}$$ where $|\varphi|$ is the cyclically reduced word norm of $\varphi$ in terms of $T_A^{\pm}$ and $T_B^{\pm}$ and $K = \log {\sqrt{\mu} + \sqrt{4 + \mu} \over 2}.$
	
	Moreover, when $\langle T_A, T_B \rangle \cong F_2$, we further have $$\frac{\log |\varphi|}{4 \log \lambda} \le k \le {K |\varphi| \over \log \lambda}.$$
\end{restatable}

The paper is organized as follows. In Section \ref{sec:thurstonconstruction}, we review Thurston's construction of pseudo-Anosov mapping classes. We present Leininger's characterization in \cite{leininger2004groups} of filling multicurves $A$ and $B$ whose multitwists generate a free group $F_2$ of rank 2 and its typicality. We also classify elements in $\langle T_A, T_B \rangle \cong F_2$ using the representation obtained from Thurston's construction.

In Section \ref{sec:asympbehav}, we investigate the dynamical properties arising from (typical) Thurston's construction. After reviewing the theory of topological entropy, we present a spectral theorem of Dahmani and Horbez. For the sake of completeness, we prove a version of the theorem asserted in \cite{dahmani2018spectral} in the setting of Thurston's construction. We also investigate the eventually pseudo-Anosov behavior of random walks without moment condition.


Section \ref{sec:appmat} deals with applications of our results. Following Kojima and McShane, we consider the hyperbolic volume of a mapping torus with pseudo-Anosov monodromy and stretch factors from Thurston's construction. While our strategy relies on probabilistic tools, we successfully obtain a non-probabilistic estimate regarding the stretch factors.

Let us finish the introduction by suggesting further questions: \begin{enumerate}
	\item Is there a generalization of Thurston's construction involving $n$ multicurves $A_{1}, \ldots, A_{n}$ so that $\langle T_{A_{1}}, \ldots, T_{A_{n}}\rangle \cong F_{n}$ in a typical case?

	\item On which subgroup of $\Mod(S)$ almost every random walk associated with it is eventually pseudo-Anosov, as in Theorem \ref{theorem:C}?
	
	\item Does the spectral theorem hold under the finite $p$-th moment condition with respect to the Teichm\"uller metric, for $1 \le p < 2$? Or, can we provide the threshold for $p$?
\end{enumerate}

\subsection*{Acknowledgments}
We truly appreciate Mladen Bestvina, Ilya Gekhtman, Camille Horbez, Paul Jung, Wanmo Kang, Sang-hyun Kim, Hidetoshi Masai, Bram Petri, and Giulio Tiozzo for fruitful conversations. Finally we thank the anonymous referees for their valuable comments which improved the organization of the paper a lot. 
The first and second authors were partially supported by Samsung Science \& Technology Foundation grant No. SSTF-BA1702-01. 
The third author was supported by KAIST Undergraduate Research Participation
program.

%
%

\medskip
\section{Thurston's construction}	\label{sec:thurstonconstruction}

\subsection{Quick review of Thurston's construction}

In this section, we review Thurston's construction of pseudo-Anosov mapping classes using multitwists along filling multicurves. Throughout the paper, every surface is a  closed orientable connected surface of genus $g > 1$; every simple closed curve is essential, i.e., not homotopic to a single point, unless otherwise stated.

\begin{definition}[Filling multicurves]
	A \emph{multicurve} $A$ on a surface $S$ is a finite set of disjoint simple closed curves on $S$. A pair $(A, B)$ of multicurves $A$ and $B$ is \emph{filling} if $S \setminus \bigcup_{\gamma \in A \cup B} \gamma$ is a disjoint union of open disks.
\end{definition}

Note that in the above definition of a multicurve, a multicurve is allowed to have two isotopic simple closed curves as distinct elements.

\begin{definition}[Multitwist]
	For a multicurve $A$ on a surface $S$, the \emph{multitwist} $T_A$ along $A$ is the product $$T_A := \prod _{\alpha \in A} T_{\alpha}$$ where $T_{\alpha}$ is the Dehn twist along a simple closed curve $\alpha \in A$.
\end{definition}

Given a filling pair $(A, B)$ of multicurves, Thurston's construction provides a $\PR$-representation of $\langle T_A, T_B \rangle$. We can then apply the classification of elements in $\PR$ to deduce that in $\langle T_A, T_B \rangle$.

We now explain Thurston's construction. The first ingredient is as follows:

\begin{lem}[Perron--Frobenius]
	Let $M$ be an $n \times n$ matrix with integer entries. Assume $M$ is primitive, i.e., $M$ is nonnegative and $M^m$ is positive for some $m \in \N$. Then $M$ has a unique nonnegative unit eigenvector $v$. Moreover, $v$ is positive and has a positive eigenvalue, called the \emph{Perron--Frobenius eigenvalue}, which is larger in absolute value than all other eigenvalues.
\end{lem}

Let $A = \{\alpha_{1}, \ldots, \alpha_{n}\}$ and $B = \{\beta_{1}, \cdots, \beta_{m}\}$. We define a matrix $N$ by $$N_{jk} = i(\alpha_j, \beta_k)$$ where $i(\cdot, \cdot)$ is the geometric intersection number. It can be shown that $NN^t$ is primitive, and thus by Perron--Frobenius theorem, its Perron--Frobenius eigenvalue $\mu = \mu(A, B)$ is obtained. Together with this observation, Thurston constructed pseudo-Anosov mapping classes in the following theorem.

\begin{thm}[\cite{thurston1988geometry}, {\cite[Chapter 14]{FarbMargalit12}}] \label{thm:thurstonconstruction}
	
	Let $(A, B)$ be a filling pair of multicurves on $S$. For the Perron--Frobenius eigenvalue $\mu = \mu(A, B)$ obtained by the above procedure, there is a representation $\rho : \langle T_A, T_B \rangle \to \PR$ defined by $$T_A \mapsto \begin{bmatrix}
	1 & -\sqrt{\mu} \\ 0 & 1
	\end{bmatrix}\quad T_B \mapsto \begin{bmatrix}
	1 & 0 \\ \sqrt{\mu} & 1
	\end{bmatrix}.$$
	
	The representation $\rho$ has the following properties:
	
	\begin{enumerate}
		\item For $\varphi \in \langle  T_A, T_B \rangle$, $\varphi$ is periodic, reducible, or pseudo-Anosov if $\rho(\varphi)$ is elliptic, parabolic, or hyperbolic, respectively.
		
		\item If $\rho(\varphi)$ is parabolic, then $\varphi$ is a multitwist.
		
		\item If $\rho(\varphi)$ is hyperbolic, the largest eigenvalue of $\rho(\varphi)$ is the stretch factor of the pseudo-Anosov mapping class $\varphi$.
	\end{enumerate}
	
\end{thm}

\subsection{Group-theoretic aspects of Thurston's construction} \label{subsec:groupaspect}

We now connect the group-theoretic aspects of $\langle  T_A, T_B \rangle$ with the topological aspects of multicurves $A$ and $B$.

The crucial observation on Thurston's construction is that $\langle T_A , T_B \rangle \cong F_2$, a free group of rank 2, for a typical configuration of $A$ and $B$. To see the typicality, we first follow the work of Leininger \cite{leininger2004groups} that provides a necessary and sufficient condition for $A$ and $B$ to result in $\langle T_A, T_B \rangle \cong F_2$. We begin by defining a graph associated with multicurves $A$ and $B$:

\begin{definition}[Configuration graph, \cite{leininger2004groups}]
	For a finite set $C$ of simple closed curves on $S$, the \emph{configuration graph} $G(C)$ for $C$ is a graph defined by the following rule: \begin{itemize}
		\item The vertex set of $G(C)$ is $C$ and;
		\item There are $i(c_1, c_2)$ edges between vertices corresponding to $c_1, c_2 \in C$.
	\end{itemize}
\end{definition}

	\begin{figure}[h]
		
	\begin{tikzpicture}[scale=1.7, every node/.style={scale=1}]
	\draw (-2, 0) .. controls (-2, 1) and (-1, 0.5) .. (0, 0.5) .. controls (1, 0.5) and (2, 1) .. (2, 0) .. controls (2, -1) and (1, -0.5) .. (0, -0.5) .. controls (-1, -0.5) and (-2, -1) .. (-2, 0);
	
	\draw (-1.5, 0.1) .. controls (-1.2, -0.1) and (-0.8, -0.1) .. (-0.5, 0.1);
	\draw (1.5, 0.1) .. controls (1.2, -0.1) and (0.8, -0.1) .. (0.5, 0.1);
	
	\draw (-1.2, 0) .. controls (-1.1, 0.1) and (-0.9, 0.1) .. (-0.8, 0);
	\draw (1.2, 0) .. controls (1.1, 0.1) and (0.9, 0.1) .. (0.8, 0);
	
	\draw[red] (-1, -0.05) .. controls (-1.1, -0.3) .. (-1, -0.6);
	\draw[red, dashed] (-1, -0.05) .. controls (-0.9, -0.3) .. (-1, -0.6);
	
	\draw[red, dashed] (1, -0.05) .. controls (1.1, -0.3) .. (1, -0.6);
	\draw[red] (1, -0.05) .. controls (0.9, -0.3) .. (1, -0.6);
	
	\draw[red, dashed] (1, 0.08) .. controls (1.1, 0.3) .. (1, 0.6);
	\draw[red] (1, 0.08) .. controls (0.9, 0.3) .. (1, 0.6);
	
	\draw[red] (0, 0.5) .. controls (-0.2, 0) .. (0, -0.5);
	\draw[red, dashed] (0, 0.5) .. controls (0.2, 0) .. (0, -0.5);
	
	\draw[blue] (-1.7, 0) .. controls (-1.7, 0.7) and (-1, 0.3) .. (0, 0.3) .. controls (1, 0.3) and (1.7, 0.7) .. (1.7, 0) .. controls (1.7, -0.7) and (1, -0.3) .. (0, -0.3) .. controls (-1, -0.3) and (-1.7, -0.7) .. (-1.7, 0);
	
	\begin{scope}[shift={(0, 0.1)}]
	\draw (3, 0) -- (4, 0) -- (5, 0.45);
	\draw (4, 0) -- (5, -0.45);
	\draw (4, 0) .. controls (3.8, -0.25) and (3.8, -0.6) .. (4, -0.85);
	\draw (4, 0) .. controls (4.2, -0.3) and (4.2, -0.5) .. (4, -0.85);
	
	\filldraw[red] (3, 0) circle(1.5pt);
	
	\filldraw[blue] (4, 0) circle(1.5pt);
	
	\filldraw[red] (4, -0.85) circle(1.5pt);
	
	\filldraw[red] (5, 0.45) circle(1.5pt);
	\filldraw[red] (5, -0.45) circle(1.5pt);
	\end{scope}
	\end{tikzpicture}
	\caption{Filling multicurves $A$, $B$ on $S$ and the configuration graph $G(A \sqcup B)$.} \label{fig:exofconfigurationgraph}
	\end{figure}
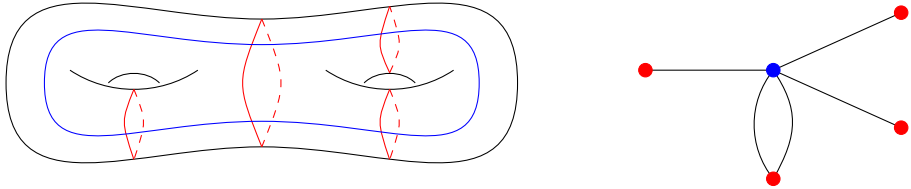

With this definition, Leininger proved the following theorem.

\begin{thm}[Leininger, \cite{leininger2004groups}] \label{thm:leininger}
	For a filling pair $(A, B)$ of multicurves, $$\langle T_A, T_B \rangle \ncong F_2 \Leftrightarrow \begin{matrix} G(A \sqcup B) \mbox{ is one of} \\ \mbox{the graphs in Figure \ref{fig:treelike}}: \end{matrix}$$
	
	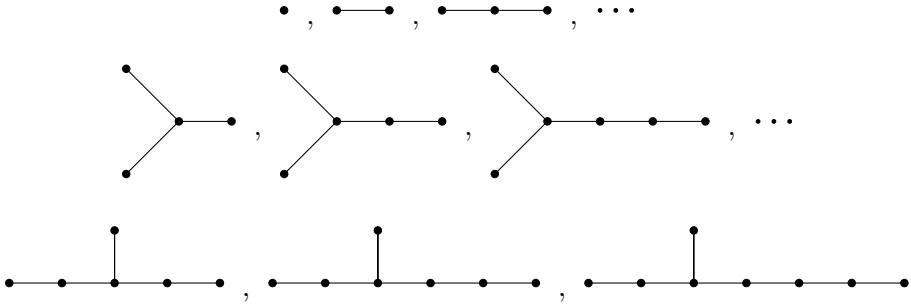
\begin{figure}[h]
		\begin{tikzpicture}[scale=0.7]

		\filldraw (0, 0) circle(2pt);
		
		\draw (0.5, 0) node[below] {$,$};
		
		\draw (1, 0) -- (2, 0);	
		\filldraw (1, 0) circle(2pt);
		\filldraw (2, 0) circle(2pt);
		
		\draw (2.5, 0) node[below] {$, $};
		
		\draw (3, 0) -- (4, 0) -- (5, 0);
		\filldraw (3, 0) circle(2pt);
		\filldraw (4, 0) circle(2pt);
		\filldraw (5, 0) circle(2pt);
		
		\draw (5.5, 0) node[below] {$,$};
		\filldraw (6, 0) circle(1pt);
		\filldraw (6.3, 0) circle(1pt);
		\filldraw (6.6, 0) circle(1pt);
		
		\draw (0, -0.5);
		
		\end{tikzpicture}
		
		\begin{tikzpicture}[scale=0.7]
		
		\draw (-2, 1) -- (-1, 0) -- (-2, -1);
		\draw (-1, 0) -- (0, 0);
		
		\filldraw (-2, 1) circle(2pt);
		\filldraw (-2, -1) circle(2pt);
		\filldraw (-1, 0) circle(2pt);
		\filldraw (0, 0) circle(2pt);

		\draw (0.5, 0) node[below] {$,$};
		
		\draw (1, 1) -- (2, 0) -- (1, -1);
		\draw (2, 0) -- (3, 0) -- (4, 0);
		
		\filldraw (1, 1) circle(2pt);
		\filldraw (1, -1) circle(2pt);
		\filldraw (2, 0) circle(2pt);
		\filldraw (3, 0) circle(2pt);
		\filldraw (4, 0) circle(2pt);

		\draw (4.5, 0) node[below] {$,$};

		\draw (5, 1) -- (6, 0) -- (5, -1);
		\draw (6, 0) -- (7, 0) -- (8, 0) -- (9, 0);
		
		\filldraw (5, 1) circle(2pt);
		\filldraw (5, -1) circle(2pt);
		\filldraw (6, 0) circle(2pt);
		\filldraw (7, 0) circle(2pt);
		\filldraw (8, 0) circle(2pt);
		\filldraw (9, 0) circle(2pt);

		\draw (9.5, 0) node[below] {$,$};

		\filldraw (10, 0) circle(1pt);
		\filldraw (10.3, 0) circle(1pt);
		\filldraw (10.6, 0) circle(1pt);
		
		\draw (0, -1.5);
		\draw (0, 1.5);
		
		\end{tikzpicture}
		
		\begin{tikzpicture}[scale=0.7]
		
		\draw (0, 0) -- (1, 0) -- (2, 0) -- (2, 1) -- (2, 0) -- (3, 0) -- (4, 0);
		
		\filldraw (0, 0) circle(2pt);
		\filldraw (1, 0) circle(2pt);
		\filldraw (2, 0) circle(2pt);
		\filldraw (2, 1) circle(2pt);
		\filldraw (3, 0) circle(2pt);
		\filldraw (4, 0) circle(2pt);
		
		\draw (4.5, 0) node[below] {$,$};
		
		\draw (5, 0) -- (6, 0) -- (7, 0) -- (7, 1) -- (7, 0) -- (8, 0) -- (9, 0) -- (10, 0);
		
		\filldraw (5, 0) circle(2pt);
		\filldraw (6, 0) circle(2pt);
		\filldraw (7, 0) circle(2pt);
		\filldraw (7, 1) circle(2pt);
		\filldraw (8, 0) circle(2pt);
		\filldraw (9, 0) circle(2pt);
		\filldraw (10, 0) circle(2pt);
		
		\draw (10.5, 0) node[below] {$,$};
		
		\draw (11, 0) -- (12, 0) -- (13, 0) -- (13, 1) -- (13, 0) -- (14, 0) -- (15, 0) -- (16, 0) -- (17, 0);
		
		\filldraw (11, 0) circle(2pt);
		\filldraw (12, 0) circle(2pt);
		\filldraw (13, 0) circle(2pt);
		\filldraw (13, 1) circle(2pt);
		\filldraw (14, 0) circle(2pt);
		\filldraw (15, 0) circle(2pt);
		\filldraw (16, 0) circle(2pt);
		\filldraw (17, 0) circle(2pt);
		
		\draw(0, 1.5);
		
		\end{tikzpicture}
		\caption{Configuration graphs that do not entail $F_2$.} \label{fig:treelike}
	\end{figure}
\end{thm}

In other words, all configurations of $A$ and $B$ not as in Figure \ref{fig:treelike} result in $\langle T_A, T_B \rangle  \cong F_2$. Hence we observe the following typicality of $F_2$:

\typicalthurston
We now consider a quantitative model that demonstrates this typicality. Let us introduce the convention that $[A] := A/\mbox{isotopy}$ ($[B] := B/\mbox{isotopy}$, resp.), and $n_{\alpha}$ ($m_{\beta}$, resp.) is the number of representatives in $A$ ($B$, resp.) of an isotopy class $\alpha \in [A]$ ($\beta \in [B]$, resp.). We also note the result of Pankau \cite{pankau2017salem} that any nonsingular positive integer square matrix is realized as the intersection matrix of a filling pair of multicurves.

\begin{prop}[A quantitative reasoning for the typicality]  \label{prop:probmodel}
Consider the following random model for multicurves $A$ and $B$: \begin{itemize}
		\item $|[A]| = n$, $|[B]| = m$;
		\item $i(\alpha, \beta) \sim \Unif(0, 1, \ldots, k-1)$ for $\alpha \in [A]$, $\beta \in [B]$ and;
		\item $n_{\alpha}, m_{\beta} \sim \Unif(1, 2, \ldots, k)$ for $\alpha \in [A]$, $\beta \in [B]$
	\end{itemize} where the random variables are independent. Then, $$\Prob(A, B \mbox{ are filling and }\langle  T_A, T_B \rangle \ncong F_2) \le 2 - \left(1 - m/k^{m+1}\right)^n - \left(1 - n/k^{n+1}\right)^m.$$ 
\end{prop}

Fixing a surface of genus $g$, we have a restriction $n, m \le 3g-3$ and the asymptotic probability $\Prob(A, B \mbox{ are filling and } \langle T_A, T_B \rangle \ncong F_2) \to 0$ as $k \to \infty$. In other words, the probability gets closer to $0$ as we allow more intersections among the given isotopy classes. Even in an extreme case that $n = m = 1$ and $k = 4$, the above bound is calculated as $1/8$.

In order to prove the proposition, we first provide a sufficient condition for $\langle T_A, T_B \rangle \cong F_2$ in terms of $i(\cdot, \cdot)$, $n_{\alpha}$, and $m_{\beta}$. Observing that $\langle T_{\alpha}, T_{\beta}\rangle \cong F_2$ if $i(\alpha, \beta) \ge 2$ for simple closed curves $\alpha$ and $\beta$, we expect a similar phenomenon for multitwists, generalizing the case of single curves:

\begin{lem} \label{lem:suffforfree}
	
	For multicurves $A$ and $B$ on $S$, suppose that there are subsets $\emptyset \neq A' \subseteq A$ and $\emptyset \neq B' \subseteq B$ satisfying the following:
	
	\begin{enumerate}
		\item $n_{\alpha} i(\alpha, [B']) = n_{\alpha}\sum_{\beta \in [B']} i(\alpha, \beta) \ge 2$ for all $\alpha \in [ A']$
		\item $m_{\beta}i([A'], \beta) = m_{\beta}\sum_{\alpha \in [A']} i(\alpha, \beta) \ge 2$ for all $\beta \in [B']$
	\end{enumerate}
	
	\noindent Then the group $\langle  T_A, T_B \rangle$ generated by multitwists is isomorphic to  $F_2.$
	
\end{lem}

This lemma is a generalization of \cite[Theorem 3.2]{hamidi2002groups}. Proof of this lemma relies on the following ping-pong lemma:

\begin{lem}[Ping-pong lemma]
	Let $\langle  g_1, \ldots, g_n \rangle$, $n \ge 2$, be a group acting on a set $X$. Suppose that there are disjoint nonempty subsets $X_1, \ldots, X_n \subseteq~X$ such that for each $i \neq j$, $g_i^{k}(X_j) \subseteq X_i$ for all $0 \neq k \in \Z$. Then $$\langle  g_1, \ldots, g_n \rangle \cong F_n.$$
\end{lem}

\begin{proof}[Proof of Lemma \ref{lem:suffforfree}]
	To apply the ping-pong lemma, let $X$ be the set of isotopy classes of essential simple closed curves on $S$. Then for $x \in X$ and $0\neq n \in \Z$, \cite[Proposition A.1]{FLP} implies \begin{equation} \label{eqn:intersectionineq}
	\left| i(T^n_{A}x, [B']) - |n| \sum_{\alpha \in [A]} n_{\alpha} i(x, \alpha) i(\alpha, [B'])\right| \le i(x, [B']).
	\end{equation}
	
	Now let $$X_{A'} = \{x \in X : i(x, [A']) < i(x, [B'])\},$$ $$X_{B'}= \{x \in X : i(x, [B']) < i(x, [A'])\}.$$ They are nonempty since $\emptyset \neq [A'] \subseteq X_{A'}$ and $\emptyset \neq [B'] \subseteq X_{B'}$, and are disjoint. For $x \in X_{B'}$ and $0 \neq n \in \Z$, it follows from Equation \ref{eqn:intersectionineq} that $$i(T_A^{n}x, [B']) \ge |n| \sum_{\alpha \in [A]} n_{\alpha}i(x, \alpha)i(\alpha, [B']) - i(x, [B']).$$ Since $x \in X_{B'}$ and $A' \subseteq A$, $$\begin{aligned}[cl]
	|n| \sum_{\alpha \in [A]} n_{\alpha} i(x, \alpha)i(\alpha, [B']) - i(x, [B']) & >  |n| \sum_{\alpha \in [A]} n_{\alpha}i(x, \alpha)i(\alpha, [B']) - i(x, [A'])\\ & \ge |n| \sum_{\alpha \in [A']} n_{\alpha} i(x, \alpha) i(\alpha, [B']) - i(x,[A']) \\ & = \sum_{\alpha \in [A']} i(x, \alpha)(|n|n_{\alpha}i(\alpha, [B']) - 1).\end{aligned}$$ From the condition (1) given in the statement, $$\begin{aligned}\sum_{\alpha \in [A']} i(x, \alpha)(|n|n_{\alpha}i(\alpha, [B']) - 1) & \ge \sum_{\alpha \in[ A']} i(x, \alpha) = i(x, [A']) \\ & = i(T_A^nx, T_A^n[A']) = i(T_A^nx, [A'])\end{aligned}$$ which implies $$T_{A}^{n}X_{B'}\subseteq X_{A'}.$$ A similar argument after interchanging $A$ and $B$ shows that $$T_{B}^nX_{A'} \subseteq X_{B'}.$$
	
	Therefore, by the ping-pong lemma, we conclude that $$\langle  T_A, T_B \rangle \cong F_2. \qedhere$$
\end{proof}

Based on this lemma, we now prove Proposition \ref{prop:probmodel}.

\begin{proof}[Proof of Proposition \ref{prop:probmodel}]
	
We fix $|[A]| = n$ and $|[B]| = m$ as in the statement, and then denote $$\Prob(i(\alpha, \beta) = j) = p_j \quad \Prob(n_{\alpha} = j) = \Prob(m_{\beta} = j) = q_j.$$
Note that $$\begin{aligned}\Prob(A, B \mbox{ are filling}, \langle T_A, T_B \rangle \ncong F_2) & \le \Prob(n_{\alpha}i(\alpha, [B])  = 1 \mbox{ for some }\alpha \in [A])\\ & \quad + \Prob(m_{\beta}i([A], \beta) = 1 \mbox{ for some }\beta \in [B]).\end{aligned}$$ We first estimate $\Prob(n_{\alpha}i(\alpha, [B]) = 1 \mbox{ for some }\alpha \in [A])$. For a fixed $\alpha \in [A]$, $$\Prob(n_{\alpha} i(\alpha, [B]) = 1) = \Prob(n_{\alpha} = 1)\Prob(i(\alpha, [B]) = 1) = q_1mp_1p_0^{m-1}$$ where the first equality comes from the independence. Hence, we have $$\Prob(n_{\alpha}i(\alpha, [B]) \neq 1 \mbox{ for all } \alpha \in [A]) = (1 - q_1mp_1p_0^{m-1})^n$$ from the independence, and thus $$\Prob(n_{\alpha}i(\alpha, [B]) = 1 \mbox{ for some }\alpha \in [A]) = 1 - (1-q_1mp_1p_0^{m-1})^n.$$ Similarly, $$\Prob(m_{\beta}i([A], \beta) = 1 \mbox{ for some }\beta \in [B]) = 1 - (1-q_1np_1p_0^{n-1})^m$$ which deduces the desired estimation $$\Prob(A, B \mbox{ are filling}, \langle T_A, T_B\rangle \ncong F_2) \le 2 - (1-q_1mp_1p_0^{m-1})^n - (1-q_1np_1p_0^{n-1})^m.$$

Now plugging in $p_{j} = q_{j} = 1/k$, we conclude $$\Prob(A, B \mbox{ are filling, }\langle  T_A, T_B \rangle \ncong F_2) \le 2 - \left(1 - m/k^{m+1}\right)^n - \left(1 - n/k^{n+1}\right)^{m}. \qedhere$$
\end{proof}

Note that the estimation above is based on a criterion weaker than Lemma \ref{lem:suffforfree}. Thus, the actual probability for not being $F_2$ will be even smaller than our estimation.

According to the classification of subgroups of $\Mod(S)$, proven by McCarthy and Papadopoulos \cite{mccarthy1989dynamics} and Kida \cite{kida2008mapping}, every subgroup of $\Mod(S)$ is classified into one of the below four types : \begin{itemize}
	\item \emph{Non-elementary subgroup}, i.e., it contains a pair of pseudo-Anosov mapping classes with disjoint fixed point sets in $\PMF(S)$.
	\item \emph{Virtually infinite cyclic subgroup}
	\item \emph{Reducible subgroup}, i.e., there exists a multicurve on $S$ fixed by an action of the subgroup
	\item \emph{Finite subgroup}
\end{itemize} This classification deduces the following proposition about $F_2$ subgroup of $\Mod(S)$ generated by two multitwists along filling multicurves.

\begin{prop} \label{prop:nonelementary}
	For a filling pair $(A, B)$ of multicurves on $S$ resulting in $\langle T_A, T_B\rangle \cong F_2$, $\langle T_A, T_B\rangle$ is a non-elementary subgroup of $\Mod(S)$.
\end{prop}

\begin{proof}
	$\langle  T_A, T_B \rangle \cong F_{2}$ is clearly not a finite subgroup of $\Mod(S)$. Now suppose that $\langle T_A, T_B \rangle$ fixes a multicurve $C$. Let $M$ be the number of curves in $C$, and let $\gamma$ be one such curve.  Since $T_A$ permutes the curves in $C$, we observe $T_A^{M!}(\gamma) = \gamma$. This implies that $A$ and $\gamma$ are disjoint, possibly after perturbing $A$ a bit. Similarly, $T_{B}^{M!}(\gamma) = \gamma$, implying that $B$ and $\gamma$ are disjoint. Thus, a complementary region of $A \cup B$ contains the essential curve $\gamma$; that region cannot be null-homotopic, which contradicts the filling condition of $(A, B)$. In conclusion, $\langle T_A, T_B \rangle$ cannot be reducible.
	
	Now suppose that $\langle T_A, T_B \rangle$ is a virtually infinite cyclic subgroup, i.e., $$\langle T_A, T_B \rangle = \bigsqcup_{i = 1}^{n} h_i \langle  g \rangle$$ for some $g, h_i \in \langle T_A, T_B\rangle$ and $n \in \N$. Let $H, T \in \{T_A, T_B, T_A^{-1}, T_B^{-1}\}$ be the first and the last alphabet of the reduced word representing $g$, respectively. Furthermore, let $N = 1 + \max_{i} |h_i|$ where $| \cdot |$ denotes the word norm in terms of $T_A^{\pm}$ and $T_B^{\pm}$. Finally, let $R \in \{T_A, T_B, T_A^{-1}, T_B^{-1}\} \setminus \{H^{-1}, T\}$. Then $$R^N = h_j g^k = h_j (H\cdots T)^k$$holds for some $j \in \{1, \ldots, n\}$ and $k \in \Z$ since $R^N \in \langle  T_A, T_B \rangle = \bigsqcup_{i = 1}^{n} h_i \langle  g \rangle$.
	
	If $k = 0$, then $R^N = h_j$, deducing $|h_j| = N > |h_j|$, contradiction.
	
	If $k > 0$, then $$R^N T^{-1} \cdots H^{-1} = h_j$$ where the left-hand side is in its reduced form. Hence, it follows that $|h_j| = |R^N T^{-1} \cdots H^{-1}| \ge N > |h_j|$, a contradiction.
	
	If $k < 0$, then $$R^N H \cdots T = h_j$$ where the left-hand side is in its reduced form. Hence, $|h_j| = |R^N H \cdots T| \ge N > |h_j|$, a contradiction.
	
	Therefore, $\langle T_A, T_B \rangle \cong F_2$ cannot be virtually infinite cyclic subgroup and we conclude that it is a non-elementary subgroup.
\end{proof}

\subsection{Elements of $\langle  T_A, T_B \rangle \cong F_2$} \label{subsec:classificationF_2}

We finish this section by classifying elements of $\langle  T_A, T_B \rangle \cong F_2$ for a filling pair $(A, B)$. First, we observe the following property of the Perron--Frobenius eigenvalue $\mu = \mu(A, B)$:

\begin{lem} \label{lem:necforfree}
	For filling multicurves $A$ and $B$ on $S$, with $\langle  T_A, T_B \rangle \cong F_2$, $$\mu = \mu(A, B) \ge 4.$$
\end{lem}

\begin{proof}
	Suppose to the contrary that $\mu < 4$. Then, $$ \tr \left(\rho(T_A)\rho(T_B) \right)  = 2 - \mu \in (-2, 2)$$ where $\rho$ is the representation for Thurston's construction. Hence $\rho(T_AT_B)$ is elliptic in $\PR$ and thus it follows from Theorem \ref{thm:thurstonconstruction}, Thurston's construction, that $T_AT_B$ is periodic, which contradicts to $\langle  T_A, T_B \rangle \cong F_2$.
\end{proof}

As an application of Lemma \ref{lem:necforfree}, we get the following classification of elements of $\langle  T_A, T_B \rangle \cong F_2$ by Leininger \cite[Proposition 6.4]{leininger2004groups}.

\begin{thm} \label{thm:classofelts}
	Let $A$ and $B$ be filling multicurves on $S$ with $\langle T_A, T_B \rangle \cong F_2$. If $\mu(A, B) > 4$, then
	$$\begin{matrix}
	\varphi \in \langle T_A, T_B \rangle \\
	\mbox{is pseudo-Anosov}
	\end{matrix} \quad \Leftrightarrow \quad
	\begin{matrix}
	\nexists g \in \langle T_A, T_B \rangle \mbox{ such that}\\
	g \varphi g^{-1} \in \langle T_A \rangle \cup \langle T_B \rangle.
	\end{matrix}$$ Otherwise (i.e., $\mu(A, B) = 4$), we have
	$$\begin{matrix}
	\varphi \in \langle T_A, T_B \rangle \\
	\mbox{is pseudo-Anosov}
	\end{matrix} \quad \Leftrightarrow \quad 
	\begin{matrix}
	\nexists g \in \langle T_A, T_B \rangle \mbox{ such that}\\
	g \varphi g^{-1} \in \langle T_A \rangle \cup \langle T_B \rangle \cup \langle T_AT_B \rangle.
	\end{matrix}$$
	
	
\end{thm}

\begin{proof}
	According to Thurston's construction, it suffices to classify elements of $\rho(\langle  T_A, T_B \rangle)$. To do this, let us consider its fundamental domain. Since $\mu \ge 4$ by Lemma \ref{lem:necforfree}, there are two cases:

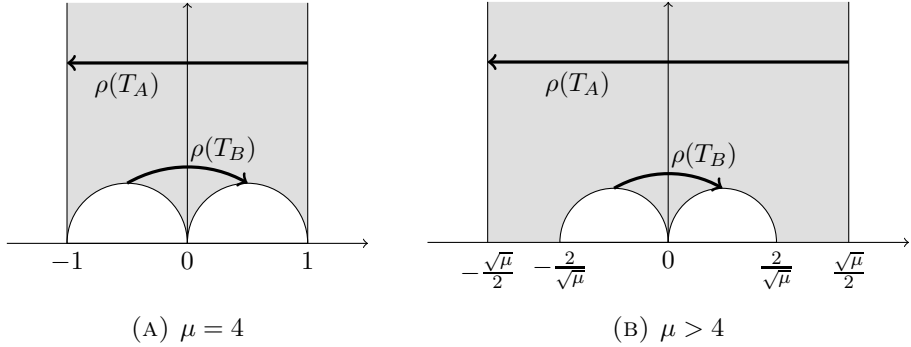
\begin{figure}[h]
	\centering
	\begin{subfigure}[b]{0.4\textwidth}
		\centering
		\begin{tikzpicture}[scale=0.8, every node/.style={scale=0.9}]
		
		\fill[gray!25] (-2, 4) -- (-2, 0) arc(180:0:1) -- (0, 4) -- (0, 4) -- (0, 0) arc(180:0:1) -- (2, 4) -- cycle;
		
		\draw[->] (-3, 0) -- (3, 0);
		\draw[->] (0, 0) -- (0, 4);
		\draw (0, 0) node[below] {$0$};
		
		\draw (-2, 0) -- (-2, 4);
		\draw (-2, 0) node[below] {$-1$};
		\draw (2, 0) -- (2, 4);
		\draw (2, 0) node[below] {$1$};
		
		\draw (0, 0) arc(0:180:1) -- cycle;
		\draw (2, 0) arc(0:180:1) -- cycle;

		\draw[very thick, ->] (2, 3) -- (-2, 3);
		\draw (-1, 3) node[below] {$\rho(T_A)$};
		
		\draw[very thick, ->] (-1, 1) arc(120:60:2);
		\draw (0.6, 1.2) node[above] {$\rho(T_B)$};
		
		\draw (0, -0.9);
		
		\end{tikzpicture}
		
		\caption{$\mu = 4$}
		\label{fig:Fund1}
	\end{subfigure}
	\hfill
	\begin{subfigure}[b]{0.59\textwidth}
		\centering
		\begin{tikzpicture}[scale=0.8, every node/.style={scale=0.9}]
		
		\fill[gray!25] (-3, 4) -- (-3, 0) -- (-1.8, 0) arc(180:0:0.9) -- (0, 4) -- (0, 4) -- (0, 0) arc(180:0:0.9) -- (3, 0) -- (3, 4) -- cycle;
		
		\draw[->] (-4, 0) -- (4, 0);
		\draw[->] (0, 0) -- (0, 4);
		\draw (0, 0) node[below] {$0$};
		
		\draw (-3, 0) -- (-3, 4);
		\draw (-3, 0) node[below] {$-{\sqrt{\mu} \over 2}$};
		\draw (3, 0) -- (3, 4);
		\draw (3, 0) node[below] {${\sqrt{\mu} \over 2}$};
		
		\draw (0, 0) arc(0:180:0.9) -- cycle;
		\draw (-1.8, 0) node[below] {$-{2 \over \sqrt{\mu}}$};
		
		\draw (0, 0) arc(180:0:0.9) -- cycle;
		\draw (1.8, 0) node[below] {${2 \over \sqrt{\mu}}$};

		\draw[very thick, ->] (3, 3) -- (-3, 3);
		\draw (-1.5, 3) node[below] {$\rho(T_A)$};
		
		\draw[very thick, ->] (-0.9, 0.9) arc(120:60:1.8);
		\draw (0.6, 1.08) node[above] {$\rho(T_B)$};
		
		\end{tikzpicture}
		
		\caption{$\mu > 4$}
		\label{fig:Fund2}
	\end{subfigure}
	\caption{Fundamental domains of $\rho(\langle T_A, T_B \rangle)$}	
\end{figure}

\textbf{Case 1: $\mu = 4$.} Its fundamental domain is as in Figure \ref{fig:Fund1}. Accordingly, the quotient $\mathbb{H} ^2 / \rho(\langle  T_A, T_B \rangle)$ is a thrice-punctured sphere. Since hyperbolic elements of $\rho(\langle T_A, T_B \rangle)$ correspond to loops in $\mathbb{H} ^2 / \rho(\langle  T_A, T_B \rangle)$ which are freely homotopic to closed geodesics, the theorem follows in this case.

\textbf{Case 2: $\mu > 4$.} Fundamental domain for this case is as in Figure \ref{fig:Fund2}, which implies that the quotient $\mathbb{H}^2 / \rho(\langle T_A, T_B \rangle)$ is a twice-punctured disk. Again, by the same arguments, all elements of $\rho(\langle  T_A, T_B \rangle)$ but conjugates of powers of $T_A$ and $T_B$ are hyperbolic.
\end{proof}

We explore the properties of $\langle \rho(T_{A}), \rho(T_{B})\rangle$ in detail. Given a pivot fundamental domain depicted in Figure \ref{fig:Fund1} and \ref{fig:Fund2}, we consider a tiling of $\mathbb{H}^{2}$ with its translations by $\langle \rho(T_{A}), \rho(T_{B})\rangle$ as in Figure \ref{fig:FundTiling}. This tiling has the Cayley graph $\mathcal{T}$ of $\langle \rho(T_{A}), \rho(T_{B}) \rangle \cong F_{2}$ as a dual tree (blue and red lines in the figure). Each edge of $\mathcal{T}$ is colored and directed, and each vertex is labelled with a reduced word of alphabets $\{a^{\pm}, b^{\pm}\}$. The labelling rule of vertices is as follows: if $\overrightarrow{vw}$ is the blue (red, resp.) directed edge, $w$ is represented by concatenating $a$ to $v$ from the right.


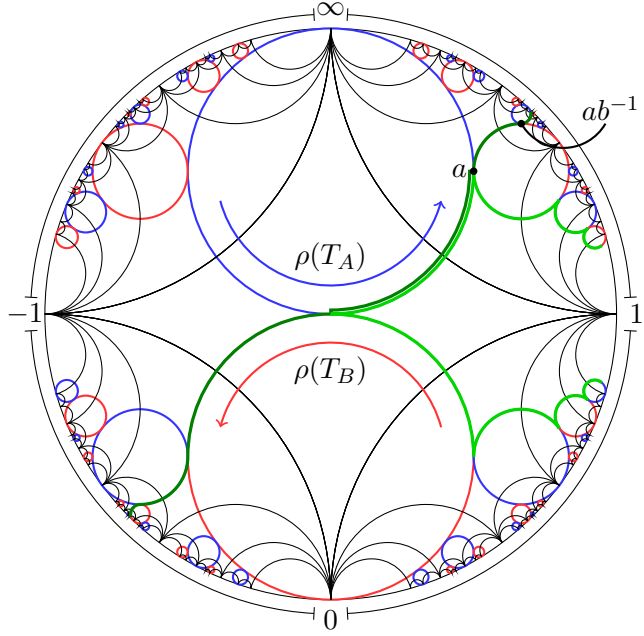
\begin{figure}[h]
\centering
\def\sc{3.8}
\begin{tikzpicture}[scale=1, every node/.style={scale=1}]
\draw (0, 0) circle (\sc);
\draw[blue!80, thick] (0, 0.5*\sc) circle (0.5*\sc);
\draw[red!80, thick] (0, -0.5*\sc) circle (0.5*\sc);

\foreach \i in {0, 1, 2, 3}{
\begin{scope}[rotate=90*\i]
\draw (\sc*1.05*0.998440764181981, \sc*1.05*0.055821504993164) arc (3.2:86.8:\sc*1.05);
\draw[rotate=3.2] (\sc*1.028, 0) -- (\sc*1.072, 0);
\draw[rotate=86.8] (\sc*1.028, 0) -- (\sc*1.072, 0);
\end{scope}
}
\draw (0, -\sc*1.07) node {$0$};
\draw (-\sc*1.07, 0) node {$-1$};
\draw (\sc*1.07, 0) node {$1$};
\draw (0, \sc*1.07) node {$\infty$};

\begin{scope}[xscale=-1, rotate=-90]

\draw[thick, red!80] (-0.5*\sc, -2*\sc/3) circle (\sc/6);
\draw[thick, blue!80] (-13*\sc*5/182, -13*\sc*12/182) circle (\sc/14);

\draw[thick, blue!80] (-29*21*\sc/870, -29*20*\sc/870) circle (\sc/30);
\draw[thick, red!80] (-17*15*\sc/306, -17*8*\sc/306) circle (\sc/18);

\draw[thick, red!80](-25*7*\sc/650, -25*24*\sc/650) circle (\sc/26);
\draw[thick, red!80](-0.433439179286761*\sc, -0.889106008793356*\sc) circle (\sc/92);

\draw[thick, blue!80](-6.5*3.3*\sc/42.9, -6.5*5.6*\sc/42.9) circle (\sc/66);
\draw[thick, blue!80](-0.99*6.5*\sc/9.7, -0.99*7.2*\sc/9.7) circle (\sc/100);

\draw[thick, red!80] (-7.3*5.5*\sc/54.02, -7.3*4.8*\sc/54.02) circle (\sc/74);
\draw[thick, blue!80] (-5.3*4.5*\sc/28.62, -5.3*2.8*\sc/28.62) circle (\sc/54);

\draw[thick, blue!80] (-8.5*7.7*\sc/73.1, -8.5*3.6*\sc/73.1) circle (\sc/86);
\draw[thick, red!80] (-3.7*3.5*\sc/14.06, -3.7*1.2*\sc/14.06) circle (\sc/38);
\draw[thick, red!80] (-1.79*11.9*\sc/30.42, -1.79*12*\sc/30.42) circle (\sc/180);

\draw (0, -\sc) arc (0:90:1*\sc);
\draw (1*\sc, 0) arc (90:180:1*\sc);
\draw (0, 1*\sc) arc (180:270:1*\sc);
\draw (-\sc, 0) arc (270:360:1*\sc);
\draw (0, -\sc) arc (0:180-36.869897645844021:\sc/3);
\draw (-0.6*\sc, -0.8*\sc) arc (-36.869897645844021:90+36.869897645844021:\sc/7);
\draw (-\sc, 0) arc (90:-90+36.869897645844021:\sc/3);

\draw (0, -\sc) arc (0:157.380135051959574:\sc/5) arc (157.380135051959574-180:151.927513064147043:\sc/21) arc (151.927513064147043-180 : 180-36.869897645844021:\sc/13);

\draw (-0.6*\sc, -0.8*\sc) arc (-36.869897645844021:136.397181027296376:\sc/17) arc (136.397181027296376-180:133.602818972703624:\sc/41) arc (133.602818972703624-180:90+36.869897645844021:\sc/17);

\draw (-\sc, 0) arc (90:90-157.380135051959574:\sc/5) arc (90-157.380135051959574+180:90-151.927513064147043:\sc/21) arc (90-151.927513064147043+180 : -90+36.869897645844021:\sc/13);

\draw (0, -\sc) arc (0: 2*81.869897645844021:\sc/7) arc (2*81.869897645844021 - 180:2*80.537677791974383:\sc/43) arc (2*80.537677791974383-180:2*78.690067525979787:\sc/31);

\draw (-5*\sc/13, -12*\sc/13) arc (2*78.690067525979787-180:2*77.471192290848489:\sc/47) arc (2*77.471192290848489-180:2*77.005383208083496:\sc/119) arc (2*77.005383208083496-180:2*75.963756532073521:\sc/55);

\draw(-8*\sc/17, -15*\sc/17) arc (2*75.96375653207352-180: 2*74.744881296942224:\sc/47) arc (2*74.744881296942224-180:2*74.054604099077145:\sc/83) arc (2*74.054604099077145-180:2*71.565051177077989:\sc/23);

\draw(-3*\sc/5, -4*\sc/5) arc (2*71.565051177077989-180:2*69.443954780416536:\sc/27) arc (2*69.443954780416536-180:2*68.962488974578183:\sc/119) arc (2*68.962488974578183-180:2*68.198590513648188:\sc/75);

\draw(-20*\sc/29, -21*\sc/29) arc (2*68.198590513648188-180:2*67.619864948040426:\sc/99) arc (2*67.619864948040426-180:270-2*67.619864948040426:\sc/239) arc (90-2*67.619864948040426:270-2*68.198590513648188:\sc/99);

\draw (-21*\sc/29, -20*\sc/29)  arc (90-2*68.198590513648188:270-2*68.962488974578183:\sc/75) arc (90-2*68.962488974578183:270-2*69.443954780416536:\sc/119) arc (90-2*69.443954780416536:270-2*71.565051177077989:\sc/27);

\draw(-4*\sc/5, -3*\sc/5) arc (90-2*71.565051177077989:270-2*74.054604099077145:\sc/23) arc (90-2*74.054604099077145:270-2*74.744881296942224:\sc/83) arc (90-2*74.744881296942224:270-2*75.96375653207352:\sc/47);

\draw(-15*\sc/17, -8*\sc/17) arc (90-2*75.963756532073521:270-2*77.005383208083496:\sc/55) arc (90-2*77.005383208083496:270-2*77.471192290848489:\sc/119) arc (90-2*77.471192290848489:270-2*78.690067525979787:\sc/47);

\draw(-12*\sc/13, -5*\sc/13) arc (90-2*78.690067525979787:270-2*80.537677791974383:\sc/31) arc (90-2*80.537677791974383:270-2*81.869897645844021:\sc/43) arc (90-2*81.869897645844021:90:\sc/7);

\end{scope}


\begin{scope}[rotate=-90]

\draw[thick, red!80] (-0.5*\sc, -2*\sc/3) circle (\sc/6);
\draw[thick, blue!80] (-13*\sc*5/182, -13*\sc*12/182) circle (\sc/14);

\draw[thick, blue!80] (-29*21*\sc/870, -29*20*\sc/870) circle (\sc/30);
\draw[thick, red!80] (-17*15*\sc/306, -17*8*\sc/306) circle (\sc/18);

\draw[thick, red!80](-25*7*\sc/650, -25*24*\sc/650) circle (\sc/26);
\draw[thick, red!80](-0.433439179286761*\sc, -0.889106008793356*\sc) circle (\sc/92);

\draw[thick, blue!80](-6.5*3.3*\sc/42.9, -6.5*5.6*\sc/42.9) circle (\sc/66);
\draw[thick, blue!80](-0.99*6.5*\sc/9.7, -0.99*7.2*\sc/9.7) circle (\sc/100);

\draw[thick, red!80] (-7.3*5.5*\sc/54.02, -7.3*4.8*\sc/54.02) circle (\sc/74);
\draw[thick, blue!80] (-5.3*4.5*\sc/28.62, -5.3*2.8*\sc/28.62) circle (\sc/54);

\draw[thick, blue!80] (-8.5*7.7*\sc/73.1, -8.5*3.6*\sc/73.1) circle (\sc/86);
\draw[thick, red!80] (-3.7*3.5*\sc/14.06, -3.7*1.2*\sc/14.06) circle (\sc/38);
\draw[thick, red!80] (-1.79*11.9*\sc/30.42, -1.79*12*\sc/30.42) circle (\sc/180);

\draw (0, -\sc) arc (0:90:1*\sc);
\draw (1*\sc, 0) arc (90:180:1*\sc);
\draw (0, 1*\sc) arc (180:270:1*\sc);
\draw (-\sc, 0) arc (270:360:1*\sc);
\draw (0, -\sc) arc (0:180-36.869897645844021:\sc/3);
\draw (-0.6*\sc, -0.8*\sc) arc (-36.869897645844021:90+36.869897645844021:\sc/7);
\draw (-\sc, 0) arc (90:-90+36.869897645844021:\sc/3);

\draw (0, -\sc) arc (0:157.380135051959574:\sc/5) arc (157.380135051959574-180:151.927513064147043:\sc/21) arc (151.927513064147043-180 : 180-36.869897645844021:\sc/13);

\draw (-0.6*\sc, -0.8*\sc) arc (-36.869897645844021:136.397181027296376:\sc/17) arc (136.397181027296376-180:133.602818972703624:\sc/41) arc (133.602818972703624-180:90+36.869897645844021:\sc/17);

\draw (-\sc, 0) arc (90:90-157.380135051959574:\sc/5) arc (90-157.380135051959574+180:90-151.927513064147043:\sc/21) arc (90-151.927513064147043+180 : -90+36.869897645844021:\sc/13);

\draw (0, -\sc) arc (0: 2*81.869897645844021:\sc/7) arc (2*81.869897645844021 - 180:2*80.537677791974383:\sc/43) arc (2*80.537677791974383-180:2*78.690067525979787:\sc/31);

\draw (-5*\sc/13, -12*\sc/13) arc (2*78.690067525979787-180:2*77.471192290848489:\sc/47) arc (2*77.471192290848489-180:2*77.005383208083496:\sc/119) arc (2*77.005383208083496-180:2*75.963756532073521:\sc/55);

\draw(-8*\sc/17, -15*\sc/17) arc (2*75.96375653207352-180: 2*74.744881296942224:\sc/47) arc (2*74.744881296942224-180:2*74.054604099077145:\sc/83) arc (2*74.054604099077145-180:2*71.565051177077989:\sc/23);

\draw(-3*\sc/5, -4*\sc/5) arc (2*71.565051177077989-180:2*69.443954780416536:\sc/27) arc (2*69.443954780416536-180:2*68.962488974578183:\sc/119) arc (2*68.962488974578183-180:2*68.198590513648188:\sc/75);

\draw(-20*\sc/29, -21*\sc/29) arc (2*68.198590513648188-180:2*67.619864948040426:\sc/99) arc (2*67.619864948040426-180:270-2*67.619864948040426:\sc/239) arc (90-2*67.619864948040426:270-2*68.198590513648188:\sc/99);

\draw (-21*\sc/29, -20*\sc/29)  arc (90-2*68.198590513648188:270-2*68.962488974578183:\sc/75) arc (90-2*68.962488974578183:270-2*69.443954780416536:\sc/119) arc (90-2*69.443954780416536:270-2*71.565051177077989:\sc/27);

\draw(-4*\sc/5, -3*\sc/5) arc (90-2*71.565051177077989:270-2*74.054604099077145:\sc/23) arc (90-2*74.054604099077145:270-2*74.744881296942224:\sc/83) arc (90-2*74.744881296942224:270-2*75.96375653207352:\sc/47);

\draw(-15*\sc/17, -8*\sc/17) arc (90-2*75.963756532073521:270-2*77.005383208083496:\sc/55) arc (90-2*77.005383208083496:270-2*77.471192290848489:\sc/119) arc (90-2*77.471192290848489:270-2*78.690067525979787:\sc/47);

\draw(-12*\sc/13, -5*\sc/13) arc (90-2*78.690067525979787:270-2*80.537677791974383:\sc/31) arc (90-2*80.537677791974383:270-2*81.869897645844021:\sc/43) arc (90-2*81.869897645844021:90:\sc/7);

\end{scope}


\begin{scope}[xscale = -1, rotate=90]

\draw[thick, blue!80] (-0.5*\sc, -2*\sc/3) circle (\sc/6);
\draw[thick, red!80] (-13*\sc*5/182, -13*\sc*12/182) circle (\sc/14);

\draw[thick, red!80] (-29*21*\sc/870, -29*20*\sc/870) circle (\sc/30);
\draw[thick, blue!80] (-17*15*\sc/306, -17*8*\sc/306) circle (\sc/18);

\draw[thick, blue!80](-25*7*\sc/650, -25*24*\sc/650) circle (\sc/26);
\draw[thick, blue!80](-0.433439179286761*\sc, -0.889106008793356*\sc) circle (\sc/92);

\draw[thick, red!80](-6.5*3.3*\sc/42.9, -6.5*5.6*\sc/42.9) circle (\sc/66);
\draw[thick, red!80](-0.99*6.5*\sc/9.7, -0.99*7.2*\sc/9.7) circle (\sc/100);

\draw[thick, blue!80] (-7.3*5.5*\sc/54.02, -7.3*4.8*\sc/54.02) circle (\sc/74);
\draw[thick, red!80] (-5.3*4.5*\sc/28.62, -5.3*2.8*\sc/28.62) circle (\sc/54);

\draw[thick, red!80] (-8.5*7.7*\sc/73.1, -8.5*3.6*\sc/73.1) circle (\sc/86);
\draw[thick, blue!80] (-3.7*3.5*\sc/14.06, -3.7*1.2*\sc/14.06) circle (\sc/38);
\draw[thick, blue!80] (-1.79*11.9*\sc/30.42, -1.79*12*\sc/30.42) circle (\sc/180);

\draw (0, -\sc) arc (0:90:1*\sc);
\draw (1*\sc, 0) arc (90:180:1*\sc);
\draw (0, 1*\sc) arc (180:270:1*\sc);
\draw (-\sc, 0) arc (270:360:1*\sc);
\draw (0, -\sc) arc (0:180-36.869897645844021:\sc/3);
\draw (-0.6*\sc, -0.8*\sc) arc (-36.869897645844021:90+36.869897645844021:\sc/7);
\draw (-\sc, 0) arc (90:-90+36.869897645844021:\sc/3);

\draw (0, -\sc) arc (0:157.380135051959574:\sc/5) arc (157.380135051959574-180:151.927513064147043:\sc/21) arc (151.927513064147043-180 : 180-36.869897645844021:\sc/13);

\draw (-0.6*\sc, -0.8*\sc) arc (-36.869897645844021:136.397181027296376:\sc/17) arc (136.397181027296376-180:133.602818972703624:\sc/41) arc (133.602818972703624-180:90+36.869897645844021:\sc/17);

\draw (-\sc, 0) arc (90:90-157.380135051959574:\sc/5) arc (90-157.380135051959574+180:90-151.927513064147043:\sc/21) arc (90-151.927513064147043+180 : -90+36.869897645844021:\sc/13);

\draw (0, -\sc) arc (0: 2*81.869897645844021:\sc/7) arc (2*81.869897645844021 - 180:2*80.537677791974383:\sc/43) arc (2*80.537677791974383-180:2*78.690067525979787:\sc/31);

\draw (-5*\sc/13, -12*\sc/13) arc (2*78.690067525979787-180:2*77.471192290848489:\sc/47) arc (2*77.471192290848489-180:2*77.005383208083496:\sc/119) arc (2*77.005383208083496-180:2*75.963756532073521:\sc/55);

\draw(-8*\sc/17, -15*\sc/17) arc (2*75.96375653207352-180: 2*74.744881296942224:\sc/47) arc (2*74.744881296942224-180:2*74.054604099077145:\sc/83) arc (2*74.054604099077145-180:2*71.565051177077989:\sc/23);

\draw(-3*\sc/5, -4*\sc/5) arc (2*71.565051177077989-180:2*69.443954780416536:\sc/27) arc (2*69.443954780416536-180:2*68.962488974578183:\sc/119) arc (2*68.962488974578183-180:2*68.198590513648188:\sc/75);

\draw(-20*\sc/29, -21*\sc/29) arc (2*68.198590513648188-180:2*67.619864948040426:\sc/99) arc (2*67.619864948040426-180:270-2*67.619864948040426:\sc/239) arc (90-2*67.619864948040426:270-2*68.198590513648188:\sc/99);

\draw (-21*\sc/29, -20*\sc/29)  arc (90-2*68.198590513648188:270-2*68.962488974578183:\sc/75) arc (90-2*68.962488974578183:270-2*69.443954780416536:\sc/119) arc (90-2*69.443954780416536:270-2*71.565051177077989:\sc/27);

\draw(-4*\sc/5, -3*\sc/5) arc (90-2*71.565051177077989:270-2*74.054604099077145:\sc/23) arc (90-2*74.054604099077145:270-2*74.744881296942224:\sc/83) arc (90-2*74.744881296942224:270-2*75.96375653207352:\sc/47);

\draw(-15*\sc/17, -8*\sc/17) arc (90-2*75.963756532073521:270-2*77.005383208083496:\sc/55) arc (90-2*77.005383208083496:270-2*77.471192290848489:\sc/119) arc (90-2*77.471192290848489:270-2*78.690067525979787:\sc/47);

\draw(-12*\sc/13, -5*\sc/13) arc (90-2*78.690067525979787:270-2*80.537677791974383:\sc/31) arc (90-2*80.537677791974383:270-2*81.869897645844021:\sc/43) arc (90-2*81.869897645844021:90:\sc/7);

\end{scope}


\begin{scope}[rotate=90]

\draw[thick, blue!80] (-0.5*\sc, -2*\sc/3) circle (\sc/6);
\draw[thick, red!80] (-13*\sc*5/182, -13*\sc*12/182) circle (\sc/14);

\draw[thick, red!80] (-29*21*\sc/870, -29*20*\sc/870) circle (\sc/30);
\draw[thick, blue!80] (-17*15*\sc/306, -17*8*\sc/306) circle (\sc/18);

\draw[thick, blue!80](-25*7*\sc/650, -25*24*\sc/650) circle (\sc/26);
\draw[thick, blue!80](-0.433439179286761*\sc, -0.889106008793356*\sc) circle (\sc/92);

\draw[thick, red!80](-6.5*3.3*\sc/42.9, -6.5*5.6*\sc/42.9) circle (\sc/66);
\draw[thick, red!80](-0.99*6.5*\sc/9.7, -0.99*7.2*\sc/9.7) circle (\sc/100);

\draw[thick, blue!80] (-7.3*5.5*\sc/54.02, -7.3*4.8*\sc/54.02) circle (\sc/74);
\draw[thick, red!80] (-5.3*4.5*\sc/28.62, -5.3*2.8*\sc/28.62) circle (\sc/54);

\draw[thick, red!80] (-8.5*7.7*\sc/73.1, -8.5*3.6*\sc/73.1) circle (\sc/86);
\draw[thick, blue!80] (-3.7*3.5*\sc/14.06, -3.7*1.2*\sc/14.06) circle (\sc/38);
\draw[thick, blue!80] (-1.79*11.9*\sc/30.42, -1.79*12*\sc/30.42) circle (\sc/180);

\draw (0, -\sc) arc (0:90:1*\sc);
\draw (1*\sc, 0) arc (90:180:1*\sc);
\draw (0, 1*\sc) arc (180:270:1*\sc);
\draw (-\sc, 0) arc (270:360:1*\sc);
\draw (0, -\sc) arc (0:180-36.869897645844021:\sc/3);
\draw (-0.6*\sc, -0.8*\sc) arc (-36.869897645844021:90+36.869897645844021:\sc/7);
\draw (-\sc, 0) arc (90:-90+36.869897645844021:\sc/3);

\draw (0, -\sc) arc (0:157.380135051959574:\sc/5) arc (157.380135051959574-180:151.927513064147043:\sc/21) arc (151.927513064147043-180 : 180-36.869897645844021:\sc/13);

\draw (-0.6*\sc, -0.8*\sc) arc (-36.869897645844021:136.397181027296376:\sc/17) arc (136.397181027296376-180:133.602818972703624:\sc/41) arc (133.602818972703624-180:90+36.869897645844021:\sc/17);

\draw (-\sc, 0) arc (90:90-157.380135051959574:\sc/5) arc (90-157.380135051959574+180:90-151.927513064147043:\sc/21) arc (90-151.927513064147043+180 : -90+36.869897645844021:\sc/13);

\draw (0, -\sc) arc (0: 2*81.869897645844021:\sc/7) arc (2*81.869897645844021 - 180:2*80.537677791974383:\sc/43) arc (2*80.537677791974383-180:2*78.690067525979787:\sc/31);

\draw (-5*\sc/13, -12*\sc/13) arc (2*78.690067525979787-180:2*77.471192290848489:\sc/47) arc (2*77.471192290848489-180:2*77.005383208083496:\sc/119) arc (2*77.005383208083496-180:2*75.963756532073521:\sc/55);

\draw(-8*\sc/17, -15*\sc/17) arc (2*75.96375653207352-180: 2*74.744881296942224:\sc/47) arc (2*74.744881296942224-180:2*74.054604099077145:\sc/83) arc (2*74.054604099077145-180:2*71.565051177077989:\sc/23);

\draw(-3*\sc/5, -4*\sc/5) arc (2*71.565051177077989-180:2*69.443954780416536:\sc/27) arc (2*69.443954780416536-180:2*68.962488974578183:\sc/119) arc (2*68.962488974578183-180:2*68.198590513648188:\sc/75);

\draw(-20*\sc/29, -21*\sc/29) arc (2*68.198590513648188-180:2*67.619864948040426:\sc/99) arc (2*67.619864948040426-180:270-2*67.619864948040426:\sc/239) arc (90-2*67.619864948040426:270-2*68.198590513648188:\sc/99);

\draw (-21*\sc/29, -20*\sc/29)  arc (90-2*68.198590513648188:270-2*68.962488974578183:\sc/75) arc (90-2*68.962488974578183:270-2*69.443954780416536:\sc/119) arc (90-2*69.443954780416536:270-2*71.565051177077989:\sc/27);

\draw(-4*\sc/5, -3*\sc/5) arc (90-2*71.565051177077989:270-2*74.054604099077145:\sc/23) arc (90-2*74.054604099077145:270-2*74.744881296942224:\sc/83) arc (90-2*74.744881296942224:270-2*75.96375653207352:\sc/47);

\draw(-15*\sc/17, -8*\sc/17) arc (90-2*75.963756532073521:270-2*77.005383208083496:\sc/55) arc (90-2*77.005383208083496:270-2*77.471192290848489:\sc/119) arc (90-2*77.471192290848489:270-2*78.690067525979787:\sc/47);

\draw(-12*\sc/13, -5*\sc/13) arc (90-2*78.690067525979787:270-2*80.537677791974383:\sc/31) arc (90-2*80.537677791974383:270-2*81.869897645844021:\sc/43) arc (90-2*81.869897645844021:90:\sc/7);

\end{scope}

\draw[black!50!green, very thick] (0, 0) -- (0, \sc*0.014) arc (-90:0:\sc*0.486) -- (\sc*0.5, \sc*0.5) arc (180:90:\sc/6) arc (-90:0:\sc/30) arc (180:90:\sc/180);
\draw[rotate=180, black!50!green, very thick] (0, 0) arc (-90:0:\sc*0.5) arc (180:90:\sc/6) arc (-90:0:\sc/30) arc (180:90:\sc/180);
\draw[black!16!green, very thick] (\sc*0.5, \sc*0.5) arc (180:323:\sc/6) arc (143:307.5:\sc/14) arc (127.5:310:\sc/26);
\draw[black!16!green, very thick] (\sc*0.5, \sc*0.5) arc (0:-90:\sc*0.5) arc(90:0:\sc*0.5) arc (180:37:\sc/6) arc (217:52.5:\sc/14) arc (232.5:50:\sc/26);

\draw[thick, blue!80, ->] (-\sc*0.4*0.965925826289068, \sc*0.5-\sc*0.4*0.258819045102521) arc (195:345:\sc*0.4);
\draw[rotate=180, thick, red!80, ->] (-\sc*0.4*0.965925826289068, \sc*0.5-\sc*0.4*0.258819045102521) arc (195:345:\sc*0.4);
\draw (0,\sc*0.2) node {$\rho(T_{A})$};
\draw (0, -\sc*0.2) node {$\rho(T_{B})$};

\draw (0.45*\sc, 0.5*\sc) node {$a$};
\fill (0.5*\sc, 0.5*\sc) circle (0.05);
\fill (2*\sc/3, 2*\sc/3) circle (0.05);
\draw (0.98*\sc, 0.72*\sc) node {$ab^{-1}$};
\draw[thick] (2*\sc/3, 2*\sc/3) arc (210:330:\sc*0.17);

\end{tikzpicture}
\caption{Fundamental domains for $\langle \rho(T_{A}), \rho(T_{B})\rangle$ when $\mu = 4$ and the Cayley graph $\mathcal{T}$ of $F_{2}$ as a dual tree. Blue lines correspond to the (right) cosets of $\rho(T_{A})$ and red lines correspond to the (right) cosets of $\rho(T_{B})$. The invariant axis of an element in $\mathrm{Isom}(\mathcal{T})$ sending 0 to $ab^{-1}$ ($ab$, resp.) is shown as the dark (light, resp.) green line.}
\label{fig:FundTiling}
\end{figure}

We now make a 1-1 correspondence among the set $V(\mathcal{T})$ of vertices of $\mathcal{T}$, the set $\mathrm{Isom}(\mathcal{T})$ of isometries of $\mathcal{T}$, and the Deck transformation group $\langle \rho(T_A), \rho(T_B) \rangle \le \mathrm{Isom}^{+}(\mathbb{H}^2)$. For each reduced word $w$ of $\{a^{\pm}, b^{\pm}\}$, there exists a cyclically reduced word $v$ such that $w = cvc^{-1}$ for some word $c$. Then the translation $t(w)$ of $\mathcal{T}$ along the axis connecting $\{cv^{n}\}_{n \in \mathbb{Z}}$ by $v$ is the unique isometry of $\mathcal{T}$ sending $0$ to $w$. The corresponding word $w'$ of $\{T_{A}^{\pm}, T_{B}^{\pm}\}$ determines the isometry $\rho(w')$ of $\mathbb{H}^{2}$.

The axes of $t(w)$ and $\rho(w')$ are homotopic relative to $\partial \mathbb{H}^{2} = S^{1}$. When $w$ is a conjugate of a power of $a$ or $b$ ($a$, $b$, or $ab$ in the case of $\mu = 4$), the entire $\mathcal{T}$ is on the left/right side of the axis of $t(w)$, and $\rho(w')$ is parabolic. If not, $\rho(w')$ is hyperbolic. In this case, the edges of fundamental domains intersecting the axis of $t(w)$ and those intersecting the axis of $\rho(w')$ coincide.

Let us now consider a hyperbolic word $w = c v c^{-1}$ with cyclically reduced form $v = \alpha_{1} \alpha_{2} \cdots \alpha_{n}$ $(\alpha_{i}$'s are alphabets of $v$). Here, $n$ is called \emph{cyclically reduced word norm} of $w$. For each $l \in \Z$ and $k \in \{1, \ldots, n\}$, the invariant axis of $\rho(w')$ passes through the edge of the tiling that corresponds to $cv^{l} \alpha_{1} \cdots \alpha_{k}$. Let us denote this edge by $e_{ln+k}$. We then define a collection $\mathcal{J}$ of indices $j$ such that:
\begin{itemize} \item $e_{j-2} \rightarrow e_{j-1}$ was a left turn but $e_{j-1} \rightarrow e_{j}$ is a right turn, or \item$e_{j-2} \rightarrow e_{j-1}$ was a right turn but $e_{j-1} \rightarrow e_{j}$ is a left turn, or 
	\item $e_{j-1} \rightarrow e_{j}$ is a diagonal move.
\end{itemize} 
See Figure \ref{fig:moves} for details. Since $\rho(w')$ is not parabolic, $\mathcal{J}$ is not empty; it is periodic with a period $n$.

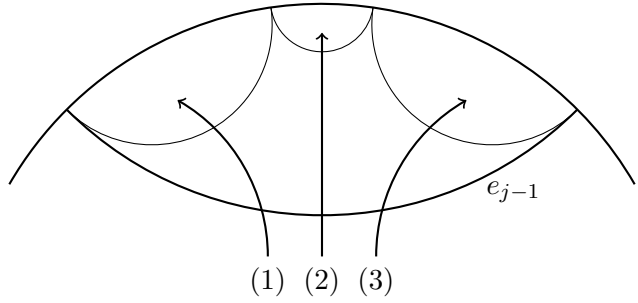
\begin{figure}[h]
	\def\sc{6}
	\begin{tikzpicture}[scale=0.8, every node/.style={scale=1}]
	\draw[rotate=30, thick] (\sc, 0) arc (0:120:\sc);
	\draw[thick, rotate=45] (\sc, 0) arc (270:180:\sc);
	\draw[rotate=225] (0, -\sc) arc (0:180-36.869897645844021:\sc/3);
	\draw[rotate=225] (-0.6*\sc, -0.8*\sc) arc (-36.869897645844021:90+36.869897645844021:\sc/7);
	\draw[rotate=225] (-\sc, 0) arc (90:-90+36.869897645844021:\sc/3);
	\draw[thick, ->] (-0.15*\sc, 0.3*\sc) arc (0:60:0.5*\sc);
	\draw[thick, ->] (0.15*\sc, 0.3*\sc) arc (180:120:0.5*\sc);
	\draw[thick, ->] (0*\sc, 0.3*\sc) -- (0*\sc, 0.92*\sc);
	\draw (0.53*\sc, 0.48*\sc) node {$e_{j-1}$};
	\draw (-0.15*\sc, 0.23*\sc) node {$(1)$};
	\draw (0*\sc, 0.23*\sc) node {$(2)$};
	\draw (0.15*\sc, 0.23*\sc) node {$(3)$};
	\end{tikzpicture}
	\caption{Consecutive edges passed through by the invariant axis. (1) left turn, (2) diagonal move, and (3) right turn.} \label{fig:moves}
\end{figure}

We now pick arbitrary $j_{0} \in \mathcal{J}$, and inductively collect $j_{i} \in \mathcal{J}$ by setting $j_{i+1} = \inf\{j \in \mathcal{J} : j \ge j_{i} + 2\}$ to form $\mathcal{J}'$. Then $\mathcal{J}'$ is also nonempty and is eventually periodic with a period $2n$. Suppose that $j_{1}, \ldots, j_{m} = j_{1} + 2n$ form a period of length $2n$.

Let us estimate $d_{\mathbb{H}^{2}}(e_{j_{k}}, e_{j_{k+1}})$ in the following two cases: \begin{enumerate}
\item $j_{k+1} = j_{k} + 2$: then $e_{j_{k}} \rightarrow e_{j_{k}+1} \rightarrow e_{j_{k+1}}$ is either a left turn after a right turn, a right turn after a left turn, or a diagonal move after any move. In either case, $d_{\mathbb{H}^2}(e_{j_{k}}, e_{j_{k+1}}) \ge \frac{1}{2}\log(3+2\sqrt{2})$.
\item $j_{k+1} \ge j_{k} + 3$: then $j \notin \mathcal{J}$ for $j_{k}+2 \le j < j_{k+1}$. This implies that $e_{j_{k}+1} \rightarrow e_{j_{k}+2}$ is either left or right turn, and $e_{j_{k}+1} \rightarrow e_{j_{k}+2} \rightarrow \cdots \rightarrow e_{j_{k+1} - 1}$ are turns of the same type, but $e_{j_{k+1} - 1} \rightarrow e_{j_{k+1}}$ is not in that type. Then the distance between $e_{j_{k}}$ and $e_{j_{k+1}}$ is at least $\frac{1}{2}\log (l + \sqrt{l^{2} - 1})$ where $l=2(j_{k+1} - j_{k} - 2) + 1 \ge j_{k+1} - j_{k}$. 
\end{enumerate}
As a result, the distance between $e_{j_{k}}$ and $e_{j_{k+1}}$ is greater than $1/2$ and ${1 \over 2}\log (j_{k+1} - j_{k})$. This gives the estimation \[
d_{\mathbb{H}^2}(e_{j_{1}}, e_{j_{m}}) \ge \sum_{k=1}^{m-1} d_{\mathbb{H}^2}(e_{j_{k}}, e_{j_{k+1}}) \ge \frac{1}{2} \sum_{k=1}^{m-1} \log (j_{k+1} - j_{k}).
\]
Since $j_{k+1} - j_k \ge 2$ and $f(x, y) =\log(x + y) - \log x - \log y$ has negative partial derivatives with $f(2, 2) = 0$, $\sum_{k = 1}^{m-1} \log (j_{k+1} - j_k) \ge \log \left( \sum_{k = 1}^{m-1} (j_{k+1} - j_k) \right)$. From $j_m = j_1 + 2n$, we conclude that \[
d_{\mathbb{H}^2}(e_{j_1}, e_{j_m}) \ge \frac{1}{2} \log (2n).
\]

Note that the translation length of $\rho(w')$ is equal to $\log \lambda_{\rho(w')}$ and is greater than $\frac{1}{2} d_{\mathbb{H}^{2}} (e_{j_1}, e_{j_{1} + 2n})$. As such, we deduce the following proposition.

\begin{prop} \label{prop:lowerboundtranslationlength}
	Let $A$ and $B$ be filling multicurves on $S$ such that $\langle T_A, T_B \rangle \cong F_2$. Then for a pseudo-Anosov $\varphi \in \langle T_A, T_B \rangle$ and its stretch factor $\lambda_{\varphi}$, we have $$\log \lambda_{\varphi} \ge \frac{1}{4} \log |\varphi|$$ where $|\cdot |$ is the cyclically reduced word norm in terms of $T_A^{\pm}$ and $T_B^{\pm}$.
\end{prop}

%
%

\medskip
\section{Asymptotic behavior of random walks obtained from Thurston's construction}	\label{sec:asympbehav}

\subsection{Introduction to the theory of topological entropy} \label{subsec:topent}

Before stating the main result, let us introduce the theory of topological entropy. The topological entropy, defined below, is a quantity representing the dynamics of a continuous map $f:X \to X$ on a compact Hausdorff space $X$.

\begin{definition}[Topological entropy, \cite{adler1965topological}] \label{def:topent}
	Let $X$ be a compact Hausdorff topological space, and $f:X \to X$ be a continuous map. For any open cover $\U$ of $X$, let us denote the minimal size of subcover of $\U$ by $N(\U)$, which is a finite integer. Then, the following limit exists: $$H(f, \U) := \lim_{n \to \infty} {1 \over n} \log  N \left( \U \vee f^{-1}(\U) \vee \cdots \vee f^{-(n-1)}(\U)\right)$$ where $\mathcal{A} \vee \mathcal{B} := \{A \cap B : A \in \mathcal{A}, B \in \mathcal{B}\}$. The \emph{topological entropy of $f$} is defined as $$h(f) := \sup_{\U} H(f, \U).$$
\end{definition}

In particular, one can consider the topological entropy of a self-homeomorphism $\varphi$ of a surface $S$. Regarding this, Thurston proved the following.

\begin{thm}[Thurston, {\cite[Expos\'e 10]{FLP}}] \label{thm:thurstonentropy}
	For a pseudo-Anosov homeomorphism $\varphi:S \to S$ of a surface $S$, $$h(\varphi) = \log \lambda_{\varphi}$$ where $\lambda_{\varphi}$ is its stretch factor. Moreover, it minimizes the topological entropy in its isotopy class.
\end{thm}

Furthermore, as $\Mod(S)$ acts on the Teichm\"uller space $\T(S)$ of $S$, one can ask the dynamics of pseudo-Anosovs on $S$ appear in $\T(S)$, in the perspective of the topological entropy. Bers' classification of mapping classes according to the action on $\T(S)$ helps to answer this question.

\begin{thm}[Bers, \cite{bers1978extremal}] \label{thm:bers}
	Let $\varphi:S \to S$ be a pseudo-Anosov homeomorphism. Then the translation length of its action on $\T(S)$ is realized, which is equal to $\log \lambda_{\varphi}$. In other words, there exists $\X_{\varphi} \in \T(S)$ such that $$\log \lambda_{\varphi} = d_{\T}(\X_{\varphi}, \varphi \cdot \X_{\varphi}) = \inf_{\mathcal{Y} \in \T(S)} d_{\T}(\mathcal{Y}, \varphi \cdot \mathcal{Y})$$ where $d_{\T}$ is the Teichm\"uller distance on $\T(S)$.
\end{thm}

Here $\X_{\varphi}$ lies on the invariant geodesic $\Gamma_{\varphi}$ of $\varphi$ in $\T(S)$, on which $\varphi$ acts  as a translation. Combining Theorem \ref{thm:thurstonentropy} and Theorem \ref{thm:bers} yields the following.

\begin{cor} \label{cor:entformula}
	Let $\varphi : S\to S$ be a pseudo-Anosov homeomorphism. Then for any $\X \in \T(S)$, we have $$h(\varphi) = \log \lambda_{\varphi} = \lim_{n \to \infty} {1 \over n } d_{\T}(\X, \varphi^{n} \cdot \X).$$
\end{cor}

\begin{proof}
	It suffices to prove the second equality, since the first equality follows from Theorem \ref{thm:thurstonentropy}. For any $\X\in \T(S)$ and $n \in \N$, $$n \log \lambda_{\varphi} = \log \lambda_{\varphi^n} \le d_{\T}(\X, \varphi^n \cdot \X)$$ by Theorem \ref{thm:bers}. Accordingly, we obtain $$\log \lambda_{\varphi} \le \liminf_{n \to \infty} {1 \over n }d_{\T}(\X, \varphi^n \cdot \X).$$
	
	To prove the opposite inequality, we observe $$d_{\T}(\X, \varphi^n \cdot \X) \le d_{\T}(\X, \X_{\varphi}) + d_{\T}(\X_{\varphi}, \varphi^n \cdot \X_{\varphi}) + d_{\T}(\varphi^n \cdot \X_{\varphi}, \varphi^n \cdot \X)$$ where $\X_{\varphi} \in \Gamma_{\varphi}$ realizes the translation length as in Theorem \ref{thm:bers}. Since $\varphi$ acts on $\T(S)$ as an isometry and acts on $\Gamma_{\varphi}$ as a translation, we have $$\begin{aligned} d_{\T}(\X, \varphi^n \cdot \X) & \le 2d_{\T}(\X, \X_{\varphi}) + n d_{\T}(\X_{\varphi}, \varphi \cdot \X_{\varphi}) \\ & = 2d_{\T}(\X, \X_{\varphi}) + n \log \lambda_{\varphi} \end{aligned}$$ Therefore, we get $$\limsup_{n \to \infty} {1 \over n}d_{\T}(\X, \varphi^n \cdot \X) \le \log \lambda_{\varphi}$$ which completes the proof.
\end{proof}

From Definition \ref{def:topent} and Corollary \ref{cor:entformula}, the topological entropy of a homeomorphism $\varphi$ indicates the dynamical behavior of the sequence of iterations $\varphi, \varphi^2, \ldots, \varphi^n, \ldots$. Recall that a pseudo-Anosov map $\varphi$ minimizes the topological entropy in its isotopy class $[\varphi]$. From now on, we define the topological entropy of $[\varphi]$ by that of $\varphi$. In this regard, the notion of topological entropy can be generalized by considering the dynamical behavior of a sequence $\varphi_1, \varphi_2, \ldots, \varphi_n, \ldots$ in $\Mod(S)$.

\begin{definition}[Topological entropy of a sequence in $\Mod(S)$, \cite{masai2018topological}]
	Let $\varphi = (\varphi_n) \in \Mod(S)^{\N}$ be a sequence of mapping classes and $\phi_n \in \mathrm{Diff}^+(S)$ be a representative of $\varphi_n$ for each $n \in \N$. Let $\phi = (\phi_n)$ be the sequence of the chosen representatives.
	
	For an open cover $\U$, we define $$H(\phi, \U) :=  \limsup_{n \to \infty}  {1 \over n}\log N\left( \U \vee \phi_1(\U) \vee \cdots \vee \phi_{n-1}(\U)\right)$$ where $N(\cdot)$ and $\vee$ are defined as in Definition \ref{def:topent}. Then we define the topological entropy of $\phi$ by $$h(\phi) := \sup_{\U} H(\phi, \U).$$ Finally, the \emph{topological entropy of $\varphi$} is defined by $$h(\varphi) := \inf_{\phi} h(\phi).$$
\end{definition}

\begin{remark}
	For a pseudo-Anosov mapping class, its iteration sequence has the same topological entropy as its pseudo-Anosov representative.
\end{remark}

With this definition, Masai proved a random version of Corollary \ref{cor:entformula}.

\begin{thm}[Masai, \cite{masai2018topological}] \label{thm:masai}
	Let $\nu$ be a probability measure on $\Mod(S)$ such that \begin{itemize}
		\item $\nu$ has a finite first moment with respect to the Teichm\"uller metric on $\T(S)$, i.e. for any $\X \in \T(S)$, $$\sum_{g \in \Mod(S)} \nu(g) d_{\T}(\X, g \cdot \X) < \infty,$$ and 
		\item $\nu$ is non-elementary, i.e. $\langle  \supp \nu \rangle$ is a non-elementary subgroup of $\Mod(S)$.
	\end{itemize} Let $\Prob$ be the probability measure on $\Mod(S)^{\N}$ induced by $\nu$. Then for $\Prob-a.e.$ $\w = (\w_n) \in \Mod(S)^{\N}$ and any $\X \in \T(S)$, $$h(\w) = \lim_{n \to \infty} {1 \over n} d_{\T}(\X, \w_n \cdot \X)$$ and it does not depend on $\w \in \Mod(S)^{\N}$. In other words, $h(\w)$ is almost surely constant.
\end{thm}

One way to investigate the dynamical property of a mapping class is to observe its action on each simple closed curve on the ambient surface. Let us equip the surface $S$ with a Riemannian metric; $l(\gamma)$ denotes the length of the geodesic representative of a closed curve $\gamma$. Then for a pseudo-Anosov $\varphi \in \Mod(S)$ and a simple closed curve $\alpha$, $$h(\varphi) = \lim_{n \to \infty} {1 \over n} \log l(\varphi^{-n}\alpha).$$ For details, see \cite[Theorem 14.23]{FarbMargalit12}. Note that this does not depend on the choice of $\alpha$ and the Riemannian metric.

The following theorem of Karlsson, together with Theorem \ref{thm:masai}, asserts that the topological entropy of a random walk has a similar property.


\begin{restatable}[Karlsson, \cite{karlsson2014two}]{thm}{Karlsson}
	There exists $\lambda$, called the \emph{Lyapunov exponent of the random walk}, such that for $\Prob-a.e.$ $\w = (\w_n) \in \Mod(S)^{\N}$, for any simple closed curve $\alpha$ and a Riemannian metric on $S$, $$\log \lambda = \lim_{n \to \infty} {1 \over n} \log l(\w_n^{-1}\alpha)$$
	holds. Moreover, for any $\X \in \T(S)$, $$\log \lambda = \lim_{n \to \infty} {1 \over n} d_{\T}(\X, \w_n \cdot \X).$$

\end{restatable}

\subsection{Spectral theorem for random walks on $\Mod(S)$}

So far, we dealt with the topological entropy of a single mapping class and that of a random walk. The former contains both topological information and dynamical information; for pseudo-Anosov $\varphi \in \Mod(S)$, $h(\varphi)$ is the log of its stretch factor and $h(\varphi) = \lim_{n \to \infty} {1 \over n} d_{\T}(\X, \varphi^n \cdot \X)$ for any $\X \in \T(S)$. The latter also describes the dynamics of a random walk in a similar way; $h(\w) = \lim_{n \to \infty} {1 \over n} d_{\T}(\X, \w_n \cdot \X)$ for any $\X \in \T(S)$ under certain condition. With this similarity, it is natural to ask how both topological entropies are related. More precisely, for a random walk $\w = (\w_n) \in \Mod(S)^{\N}$, how do $h(\w_n)$ and $h(\w)$ interact? Dahmani and Horbez answered this in the following spectral theorem when the underlying measure is finitely supported.


\DH

Together with Theorem \ref{thm:masai} and Theorem \ref{thm:thurstonentropy}, the above theorem concludes that ${1 \over n} h(\w_n) \to h(\w)$ as $n \to \infty$. In other words, the linear growth rate of the topological entropies of random mapping classes equals the topological entropy of the random walk.

Not only for finitely supported measure, Dahmani and Horbez also asserted in \cite[Remark 2.7, 3.2]{dahmani2018spectral} that the same result holds under the finite second moment condition. For the sake of completeness, we provide a detailed proof of this assertion in the setting of Thurston's construction. 

\secondmoment* \label{thm:secondmoment}

Before proving this theorem, we briefly review the Busemann boundary of hyperbolic spaces: see \cite{benoist2016central} for details. For a proper quasiconvex Gromov hyperbolic space $(M, d)$ and a reference point $\X \in M$, the {\em Busemann compactification} $M_B$ consists of functions $h_x : M \to \R$ of the form $$h_x(m) := \lim_{n \to \infty} d(m, x_n) - d(\X, x_n)$$ that are well-defined on $M$, where $x = (x_n)$ is a sequence in $M$. From the definition, we observe that $h_x(m) = d(m, x) - d(\X, x)$ for a convergent sequence $x_n \to x \in M$ and $h_x = h_y$ implies $x = y$ for $x, y \in M$. As such, it is possible to regard $M$ as a subspace of $M_B$; indeed, $M_B$ is a compactification of $M$ under a suitable topology. Then we define the {\em Busemann boundary} of $M$ as $\partial_BM := M_B \setminus M$.

We now compare the Busemann compactification $M_{B}$ of $M$ with another compactification $M_{G}$, called the {\em Gromov compactification}. Since a bounded sequence in $M$ is convergent if it defines a point in the Busemann compactification, there exists a natural map $\pi : M_B \to M_G$ which is the identity on $M$, and this is indeed surjective.

We can extend the Gromov product on $M$ to $M_B$, by defining $$(m | x)_{\X} = (x | m)_{\X} := {1 \over 2}\left(d(m, \X) - h_x(m) \right) \mbox{ and}$$ $$(x | x')_{\X} := -\inf_{m \in M} {1 \over 2} \left(h_x(m) + h_{x'}(m) \right)$$ for $m \in M$ and $x, x' \in \partial_BM$. For $x = (x_{n})$, the first formula reads $$(m | x)_{\X} = \lim_{n \to \infty} (m | x_n)_{\X}.$$

	\begin{lem}[Lemma 5.10 of \cite{maher2018random}]\label{lem:gromovApprox}
	Let $M$ be a Gromov hyperbolic space and $\X \in M$. Then there exist $K, K' > 0$ satisfying the following. For $a, b, c, d \in M$, if $(a | b)_{\X}, (c | d)_{\X} \ge A$ and $(a | c)_{\X} \le A - K$ holds for some $A$, then $|(a | c)_{\X} - (b | d)_{\X}| \le K'$.
	\end{lem}
	
We are now ready to prove Theorem \hyperref[thm:secondmoment]{A}:

\begin{proof}[Proof of the theorem]

	Thurston's construction in \cite{thurston1988geometry} describes how a Teichm{\"u}ller disk $f : \mathbb{H}^{2} \rightarrow \T(S)$ is determined from the pair $(A, B)$ (cf. \cite{leininger2004groups}, \cite{herrlich2007boundary}). Elements in $\langle T_A, T_B \rangle$ stabilize $f(\mathbb{H}^{2})$ and act on $f(\mathbb{H}^{2})$ as isometries. Thus, the isometric embedding $f$ induces a representation $\rho : \langle T_A, T_B \rangle \rightarrow \PR = \operatorname{Isom}^{+}(\mathbb{H}^{2})$. Throughout, $O(\delta)$ denotes a constant solely depending on the hyperbolicity $\delta$ of $\mathbb{H}^2$. $d( \cdot, \cdot)$ denotes the hyperbolic distance of $\mathbb{H}^{2}$, which induces a Gromov product $(\cdot | \cdot)_{\cdot}$.
	
	We now estimate the translation length of a random mapping class $\w_{n}$ viewed as an isometry of $\mathbb{H}^{2}$. From the work \cite{maher2018random} of Maher and Tiozzo, the translation length can be estimated as follows:
	
	\begin{lem}[Maher--Tiozzo] \label{lem:mttranslation}
		There exists a constant $C_0 > 0$ such that for any isometry $g$ of $\mathbb{H}^2$, if $d(\X, g \cdot \X) \ge 2(g \cdot \X | g^{-1} \cdot \X)_{\X} + C_0$, then $$\left| (\mbox{translation length of } g)- \left( d(\X, g \cdot \X) - 2(g \cdot \X | g^{-1} \cdot \X)_{\X} \right) \right| \le O(\delta).$$
	\end{lem}
	
	In view of Lemma \ref{lem:mttranslation}, we should estimate $d(\X, \w_{n}\cdot \X)$ and control the term $(\w_{n} \cdot \X | \w_{n}^{-1} \cdot \X)_{\X}$ in order to obtain the translation length of $\w_{n}$. We do this by referring to \cite[Proposition 4.1]{benoist2016central}.
	
	
	\begin{lem}[Benoist--Quint] \label{lem:bqsummable}
		Let $p > 1$, $(M, d)$ be a proper quasiconvex Gromov hyperbolic space, and $\X \in M$. Let $\nu$ be a non-elementary Borel probability measure on $G \le \mathrm{Isom^{+}}(M)$ such that $\int_G d(\X, g \cdot \X)^p d\nu(g) < \infty$.
		
		Then $\lambda := \lim_{n\to \infty} {1 \over n} \E d(\X, \w_n \cdot \X) > 0$. Moreover, for every $\varepsilon > 0$, there exist constants $C_n$ such that $\sum_{n \ge 1} n^{p-2}C_n < \infty$ and for any $x \in \partial_B M$, $$\Prob \left(|h_x(\w_n^{-1}\cdot \X) - n \lambda | > \varepsilon n\right) \le C_n, \mbox{ and}$$ $$\Prob \left(|d(\X , \w_n \cdot \X) - n \lambda| > \varepsilon n \right) \le C_n.$$

	\end{lem}
	
	Since $\langle T_A, T_B \rangle$ is discrete (\cite{thurston1988geometry}, \cite{leininger2004groups}), $\nu$ and $\hat{\nu} := \nu(\cdot^{-1})$ are Borel measures and we can apply Lemma \ref{lem:bqsummable}. Note that $\lambda$ is the drift of $\w$.
	
	For each $\varepsilon > 0$, we now obtain a summable sequence $C_{n; \varepsilon}$ such that 
	$$\Prob\left(|h_x(\w_n^{-1}\cdot \X) - n \lambda | > \varepsilon n\right) \le C_{n; \varepsilon}, \mbox{ and}$$ 
	$$\Prob \left(|d(\X , \w_n^{-1} \cdot \X) - n \lambda| > \varepsilon n\right) \le C_{n; \varepsilon}.$$ 
	Since $(\w_n^{-1} \cdot \X | x)_{\X} = {1 \over 2} \left( d(\X, \w_n^{-1} \cdot \X) - h_x(\w_n^{-1}\cdot \X) \right)$, we have 
	$$\begin{aligned} \Prob \left( (\w_n^{-1} \cdot \X | x)_{\X} > \varepsilon n \right) & \le  \Prob\left(|d(\X , \w_n^{-1} \cdot \X) - n \lambda| > \varepsilon n\right) \\
	& \quad + \Prob\left(|h_x(\w_n^{-1}\cdot \X) - n \lambda | > \varepsilon n\right) \\ & \le  2 C_{n; \varepsilon}. \end{aligned}$$ 
	A similar argument shows that $ \Prob \left( (\w_n \cdot \X | x)_{\X} > \varepsilon n \right) < 2 C_{n; \varepsilon}.$ 	
	
	Now let $0< \varepsilon < \lambda/10$. We set $m = m(n) =  \lceil n/2 \rceil$ and $u_{m(n)} := \omega_{m}^{-1} \omega_{n} = g_{m+1} \cdots g_{n}$. Note that $\omega_{m(n)}$ and $u_{m(n)}$ are independent for each $n \in \N$. Note also that $u_{m(n)}$ and $\omega_{n-m(n)}$ have the same distribution.
	
	Using the independence, we deduce\[
		\Prob \left((u_{m}^{-1} \cdot \X | \omega_{m} \cdot \X)_{\X} > \varepsilon n \right)  = \sum_{g \in \Mod(S)} \Prob(u_{m}^{-1} = g) \Prob \left((g\cdot \X | \w_{m} \cdot \X)_{\X} > \varepsilon n \right).
	\]
	If $g = id$, then $(g \cdot \X | \w_{m} \cdot \X)_{\X} = 0$. For $g \neq id$, we consider the half-geodesic $\Gamma = \overrightarrow{\X x}$ from $\X$ to some $x \in \partial_{B} \mathbb{H}^{2}$ such that $g \X \in \Gamma$. Let $\{\X_{i}\}$ be a sequence on $\Gamma$ tending to $x$, where $\X_{1} = g \X$. For each $i$, we obtain \[\begin{aligned}
	(\X_{i} | \w_{m} \cdot \X)_{\X} &= \frac{1}{2} \left[d(\X, \X_{i}) + d(\X, \w_{m} \cdot \X)  - d(\X_{i}, \w_{m} \cdot \X) \right]
	\\ &\ge  \frac{1}{2} \left[\left(d(\X, \X_{1}) + d(\X_{1}, \X_{i})\right) + d(\X, \w_{m} \cdot \X)  - \left(d(\X_{1}, \w_{m} \cdot \X) +  d(\X_{1}, \X_{i})\right)\right]
	\\ &= \frac{1}{2} \left[d(\X, \X_{1}) + d(\X, \w_{m} \cdot \X) - d(\X_{1}, \w_{m} \cdot \X) )\right] \\ & = (g \cdot\X | \w_{m} \cdot \X)_{\X}.
	\end{aligned}
	\]
	This implies that \[
	\Prob((g \cdot \X | \w_{m} \cdot \X)_{\X} > \varepsilon n) \le \Prob ((x | \w_{m} \cdot \X)_{\X} > \varepsilon n) \le 2 C_{m;\varepsilon}.
	\]
	From this and the fact that $\sum_{g \in \Mod(S)} \Prob(u_{m}^{-1} = g) = 1$, we conclude that \begin{equation}\label{eqn:busemann1}
	\Prob \left((u_{m}^{-1} \cdot \X | \omega_{m} \cdot \X)_{\X} > \varepsilon n \right) \le 1 \cdot 2C_{m;\varepsilon}.
	\end{equation}
	
	Next, we calculate \[\begin{aligned}
	& \Prob \left((\w_{m} \cdot \X | \omega_{n} \cdot \X)_{\X} <  (\lambda - \varepsilon) m \right)  \\
	& \quad = \sum_{g \in \Mod(S)} \Prob(\w_{m}= g) \Prob \left((g\cdot \X | g u_{m}  \cdot \X)_{\X} < (\lambda - \varepsilon) m \right) \\
	& \quad \le \sum_{\substack{d(\X, g \cdot \X) \\  \le (\lambda - \varepsilon/2) m}}\Prob(\w_{m}= g) +  \sum_{\substack{d(\X, g \cdot \X)\\ > (\lambda - \varepsilon/2) m}}\Prob(\w_{m}= g) \Prob \left((g\cdot \X | g u_{m}  \cdot \X)_{\X} < (\lambda - \varepsilon) m \right) \\
	& \quad = \Prob\left( d(\X, \omega_{m} \cdot \X) \le \left(\lambda - \frac{\varepsilon}{2}\right) m\right) \\
	& \quad \quad + \sum_{\substack{d(\X, g \cdot \X)\\ > (\lambda - \varepsilon/2) m}}\Prob(\w_{m}= g) \Prob \left(( \X | \w_{n-m}  \cdot \X)_{g^{-1} \X} < (\lambda - \varepsilon) m \right).
	\end{aligned}
	\]
	The first term is bounded by $C_{m;\varepsilon/2}$. To deal with the second term, we fix $g$ with $d(\X, g \cdot \X) > (\lambda - \varepsilon/2)m$. Let $\Gamma = \overrightarrow{\X x}$ be the half-geodesic from $\X$ to some $x \in \partial_{B} \mathbb{H}^{2}$  such that $g^{-1} \X \in \Gamma$. Let $\{\X_{i}\}$ be a sequence on $\Gamma$ tending to $x$, where $\X_{1} = g^{-1} \X$. If $(\X | \w_{n-m} \cdot \X)_{g^{-1} \cdot \X} <(\lambda - \varepsilon / 2) m$, then
	\[\begin{aligned}
	(\X_{i} | \omega_{n-m} \cdot \X)_{\X} &\ge (g^{-1}\cdot  \X | \omega_{n-m} \cdot \X)_{\X} \\
	& = d(g^{-1} \cdot \X, \X) - (\X | \omega_{n-m} \cdot \X)_{g^{-1} \cdot \X}\\ & \ge \left(\lambda - \frac{\varepsilon}{2}\right)m - (\lambda - \varepsilon) m \\
	& = \frac{\varepsilon}{2} m.
	\end{aligned}
	\]
	Thus, we deduce \[
	\Prob \left(( \X | \w_{n-m}  \cdot \X)_{g^{-1} \X} < (\lambda - \varepsilon) m \right) \le \Prob \left( (x | \w_{n-m} \cdot \X)_{\X} \ge \varepsilon m/2 \right) \le 2C_{m;\varepsilon/2}.
	\]
	From this we conclude that \begin{equation}\label{eqn:busemann2}
	\Prob \left((\w_{m} \cdot \X | \omega_{n} \cdot \X)_{\X} <  (\lambda - \varepsilon) m \right)  \le C_{m;\varepsilon/2} + 1 \cdot 2C_{m;\varepsilon / 2}.
	\end{equation}
	A similar argument implies  \begin{equation}\label{eqn:busemann3}
	\Prob \left((u_{m}^{-1} \cdot \X | \omega_{n}^{-1} \cdot \X)_{\X} <  (\lambda - \varepsilon) m \right)  \le 3C_{m;\varepsilon/2}.
	\end{equation}
	The probability of the events in Equation \ref{eqn:busemann1}, \ref{eqn:busemann2} and \ref{eqn:busemann3} is summable. Thus, by the Borel--Cantelli lemma, $\Prob-$a.e. $\w$ satisfies \[ \begin{aligned}
	(u_{m}^{-1} \cdot \X | \omega_{m} \cdot \X)_{\X} & \le \varepsilon n,\\
	(\w_{m} \cdot \X | \omega_{n} \cdot \X)_{\X} & \ge (\lambda - \varepsilon) m, \mbox{ and}\\
	(u_{m}^{-1} \cdot \X | \omega_{n}^{-1} \cdot \X)_{\X} & \ge  (\lambda - \varepsilon) m
	\end{aligned}\]
	for sufficiently large $n$. Then Lemma \ref{lem:gromovApprox} asserts that $(\w_{n}^{-1} \cdot \X | \w_{n} \cdot \X)_{\X} = (u_{m}^{-1} \cdot \X | \w_{m} \cdot \X)_{\X} + O(\delta) \le 2\varepsilon n$ for sufficiently large $n$.
	
	We now return to Lemma \ref{lem:mttranslation}. There exists an event $\Omega_{k}$ for $\varepsilon = \lambda/(100k)$ such that $\Prob(\Omega_{k}) = 1$ and $\w \in \Omega_{k}$ satisfies \[
	|d(\X, \w_{n} \cdot \X) - \lambda n| < \varepsilon n \mbox{ and } (\w_{n}^{-1} \cdot \X | \w_{n} \cdot \X)_{\X} \le 2\varepsilon n
	\]for sufficiently large $n$. Thus, by Lemma \ref{lem:mttranslation}, \[ \begin{aligned}
	\left|\frac{1}{n} (\mbox{translation length of } \w_{n}) - \lambda \right| & \le \frac{1}{n} \left|d(\X, g \cdot \X) +2(\w_{n}^{-1} \cdot \X | \w_{n} \cdot \X)_{\X} - n\lambda \right| \\
	& \quad + \frac{1}{n} O(\delta)\\
	& \le 5\varepsilon + \frac{1}{n} O(\delta)
	\end{aligned}
	\] for sufficiently large $n$. Then $\lim_{n} \frac{1}{n}(\mbox{translation length of } \w_{n})  = \lambda >0$ for $\w \in \Omega:=\cap_{k \in \N} \Omega_{k}$ and $\Prob(\Omega) = 1$. In particular, the translation length of $\w_{n}$ becomes positive so $\w_{n}$ is eventually pseudo-Anosov.
	\end{proof}

It is still open whether Theorem \hyperref[thm:secondmoment]{A} holds under the finite first moment condition. However, a weaker assertion does hold true as follows.

\begin{thm} \label{thm:inprobconv}
Let $(A, B)$ be a filling pair of multicurves on $S$ and $\nu : \Mod(S) \to [0, 1]$ be a non-elementary probability measure with $\supp \nu \subseteq \langle T_A, T_B \rangle$. If the first moment with respect to the Teichm\"uller metric is finite, i.e. for any $\X \in \T(S)$, $$\sum_{g \in \Mod(S)} d_{\T}(\X, g \cdot \X)\nu(g)  < \infty,$$ then as $n \to \infty$ we have $${1 \over n} \log \lambda_{\w_n} \to h(\w) \quad \mbox{in probability.}$$
\end{thm}

\begin{proof}
Since $\nu$ has finite first moment, $\lim_{n \rightarrow \infty} {1 \over n}d_{\T}(\X, \w_n \cdot \X) =h(\w)$ is almost surely equal to the drift $L_{\T} >0$ (cf. \cite{masai2018topological}, \cite[Theorem 1.2]{maher2018random}, \cite[Proposition 3.3]{benoist2016central}).

We again use the notation $m= m(n) = \lceil n/2 \rceil$ and $u_{m} := \w_{m}^{-1} \w_{n}$. The following lemma of Maher and Tiozzo holds without any moment condition.
	
	\begin{lem}[{\cite[Lemma 5.9, 5.11]{maher2018random}}]\label{lem:estimateProb}
		Let $G$ be a countable group of isometries of a separable Gromov hyperbolic space, and let $\nu$ be a non-elementary probability measure on $G$. Then for sufficiently small $\eta$,
		$$\begin{aligned} \Prob \left( (u_{m}^{-1}\cdot \X | \w_{m} \cdot \X)_{\X} > \eta n/ 2 \right) & \to 0, 
		\\ \Prob \left( ( \w_{m} \cdot \X | \w_{n} \cdot \X )_{\X} < \eta n \right) & \to 0, \mbox{ and} 
		\\ \Prob \left( (u_{m}^{-1}\cdot \X | \w_{n}^{-1} \cdot \X )_{\X} < \eta n \right) & \to 0 \end{aligned}$$ as $n \to \infty$.
	\end{lem}
	
	Suppose now that $0<\varepsilon<L_{\T}/10$ is given. We then define \[
	E_{n} =\left\{ \w : \left|\frac{1}{k} d(\X, \w_{k} \cdot \X) - L_{\T}\right| < \varepsilon/10 \mbox{ for all } k\ge n\right\}.
	\] Then $E_{n} \nearrow E = \{ \w : \limsup |\frac{1}{k} d(\X, \w_{k} \cdot \X) - L_{\T}| < \varepsilon/10\}$. Since $\frac{1}{k} d(\X, \w_{k} \cdot \X)$ tends to $L_{\T}$ almost surely, $P(E_{n}) \nearrow 1$.
	
	Next, let $\eta = \varepsilon/100$ and define\[
	\Omega_{n} = \left\{ \w  : \begin{array}{cc}(u_{m}^{-1}\cdot \X | \w_{m} \cdot \X)_{\X} \le \eta n/ 2, \quad ( \w_{m} \cdot \X | \w_{n} \cdot \X )_{\X} \ge \eta n, \\ \mbox{and } (u_{m}^{-1}\cdot \X | \w_{n}^{-1} \cdot \X )_{\X} \ge \eta n\end{array}  \right\}.
	\]
	Note that by Lemma \ref{lem:estimateProb}, $\Prob(\Omega_{n}^{c})$ tends to 0. We now let $N =1000 O(\delta) / \varepsilon$. Then for $n > N$ and $\w \in E_{n} \cap \Omega_{n}$, Lemma \ref{lem:gromovApprox} implies \[
	2(\w_{n} \cdot \X | \w_{n}^{-1} \cdot \X)_{\X} \le \eta n + 2O(\delta) \le 2\eta n < (L_{\T} - \varepsilon)n \le d(\X, \w_{n}\cdot \X).
	\]
	Thus, we can apply Lemma \ref{lem:mttranslation} and deduce that \[\begin{aligned}
	\left|\frac{1}{n} \log \lambda_{\w_{n}} - L_{\T}\right| &\le \frac{1}{n}\left| nL_{\T} -  d(\X, \w_{n} \X) -  2 (\w_{n} \cdot \X | \w_{n}^{-1} \cdot \X)_{\X} \right| + \frac{1}{n} O(\delta)
	\\ &\le (\varepsilon/10) + 2\eta + (\varepsilon/1000) < \varepsilon.
	\end{aligned}
	\]
	In conclusion, for $n>N$, \[
	\Prob \left( \left|\frac{1}{n} \log \lambda_{\w_{n}} - L_{\T}\right|  \ge \varepsilon \right) \le \Prob (E_{n}^{c} \cup  \Omega_{n}^{c}).
	\]
	Since $\Prob(E_n^{c} \cup \Omega_n^{c}) \to 0$ as $n \to \infty$, the desired convergence is proved.
\end{proof}

This in particular implies that $\Prob(\w_{n} \mbox{ is pseudo-Anosov}) \rightarrow 1$ as $n \rightarrow \infty$. This result in a general setting is due to Maher \cite{maher2011random}.

As a corollary of Theorem \ref{thm:inprobconv}, we can deduce the almost sure equality between the topological entropy and limsup of the linear growth rate of translation lengths. This was also observed by Dahmani and Horbez in \cite[Remark 2.7, 3.2]{dahmani2018spectral}. For the sake of completeness, we include a proof of the result below. 

\begin{cor}
	Let $(A, B)$ be a filling pair of multicurves on $S$ and $\nu : \Mod(S) \to [0, 1]$ be a non-elementary probability measure with $\supp \nu \subseteq \langle T_A, T_B \rangle$. If the first moment with respect to the Teichm\"uller metric is finite, then for $\Prob-a.e.$ $\w \in \Mod(S)^{\N}$, we have $$\limsup_{n \to \infty} {1 \over n} (\mbox{translation length of } \w_{n}) = h(\w)$$ where $h(\w)$ is the topological entropy of a random walk.
\end{cor}

\begin{proof}
For each $k$, we take $n_{k}$ such that \[
\Prob\left(\left|\frac{1}{n_{k}} \log \lambda_{\omega_{n_{k}}} - h(\omega) \right| > 2^{-k}\right) < 2^{-k}.
\]
The Borel--Cantelli lemma implies that $\Prob-a.e.$ $\w \in \Mod(S)^{\N}$ has a subsequence $(\w_{n_{k}})_{k > K(\w)}$ such that $|\frac{1}{n_{k}} \log \lambda_{\w_{n_{k}}} - h(\omega)| \le 2^{-k}$. Thus, it makes sense to consider the quantity $\limsup_{n} \frac{1}{n} \log \lambda_{\omega_{n}} \ge h(\omega)$ for such paths.

	The reverse inequality follows from Theorem \ref{thm:bers}: for pseudo-Anosov $\w_n$, we have $\log \lambda_{\w_n} \le d_{\T}(\X, \w_n \cdot \X)$ so $$\limsup_{n \to \infty} {1 \over n} \log \lambda_{\w_n} \le h(\w)$$ for $\Prob-a.e.$ $\w \in \Mod(S)^{\N}$, as a result of Theorem \ref{thm:masai}. 
	\end{proof}

\subsection{Eventually pseudo-Anosov behavior of random walks}

One consequence of Theorem \hyperref[thm:secondmoment]{A} is that almost every random walk associated with Thurston's construction eventually becomes pseudo-Anosov. Our purpose is to obtain this phenomenon without any moment condition.

\eventuallypA*  \label{thm:eventuallypA}

\begin{proof}
Recall Theorem \ref{thm:classofelts}: the only non-pseudo-Anosov elements that appear in $\langle T_A, T_B \rangle$ are the conjugates of powers of $T_A$, $T_B$ and $T_A T_B$.

We first prove that the conjugates of powers of $T_A$ do not appear infinitely often. Let $\{v_{1}, \ldots, v_{r}\}$ be a set of free generators of $\langle \supp \nu \rangle \le \langle T_{A}, T_B\rangle$. By \cite{hall1949coset} and \cite{burns1969free}, there exist $v_{r+1}, \ldots, v_k \in \langle T_A, T_B \rangle$ so that $v_1, \ldots, v_k$ freely generate a finite index subgroup $L$ of $\langle T_A, T_B \rangle$. Let $\{Lg_{1}, \ldots, Lg_{l}\}$ be the collection of the right cosets of $L$ in $\langle T_A, T_B \rangle$.

We define the abelianization map $Ab : L \rightarrow \Z^{k}$ sending each $v_{i}$ to $\vec{e}_{i}$ and a projection map $Q : \Z^{k} \rightarrow \Z^{r}$ discarding the last $k-r$ coordinates. Since $\langle \supp \nu \rangle \le L$, we obtain a new random walk $Q \circ Ab \circ \w$ on $\Z^{r} \subseteq \R^{r}$. For $i = 1, \ldots, l$, we pick $w_{i} \in \Z^{r} \setminus \{0\}$ such that \[
(Q \circ Ab)(\langle g_{i} T_{A}g_{i}^{-1}\rangle \cap L) \subseteq \langle w_{i} \rangle.
\]

Suppose now that $\w_{n} = gT_{A}^{j} g^{-1} \in L$ for some $\w$ and $g \in L g_{i}$. Since $\w_{n}$ and $g_{i} g^{-1}$ are in $L$, we deduce \[
(Q \circ Ab)(\w_{n}) = (Q \circ Ab)\left( (g_{i} g^{-1}) \w_{n} (g g_{i}^{-1})\right) = (Q \circ Ab)(g_{i} T_{A}^{j} g_{i}^{-1}) \in \langle w_{i} \rangle.
\]
Thus, if $\w$ visits the conjugates of powers of $T_{A}$ infinitely often, then $(Q \circ Ab \circ \w)$ visits $\langle w_{i} \rangle$ infinitely often for some $i$.

For each $i$, we consider the projection $P_{i} : \R^{r} \rightarrow \R^{r-1}$ onto the orthogonal complement of $w_{i}$. Suppose that $(P_{i} \circ Q \circ Ab)(\supp \nu) \cdot \theta = 0$ for some $\theta \in \R^{r-1} \setminus \{0\}$. This implies that $(P_{i} \circ Q \circ Ab)( \langle \supp \nu \rangle) \cdot \theta = P_{i}(\Z^{r}) \cdot \theta = 0$, which is a contradiction. Hence $P_{i} \circ Q \circ Ab \circ \w$ is truly $(r-1)$-dimensional; this implies the transience if $r \ge 4$ is assumed further (\cite{durrett2019probability}, \cite{chung2008distribution}). Consequently, we deduce\[
\Prob \left( \begin{matrix} \w_n \mbox{ visits the conjugates of} \\ \mbox{powers of } T_A \mbox{ i.o.} \end{matrix} \right) \le \sum_{i=1}^{l} \Prob\left((Q \circ Ab)(\w_{n}) \in \langle w_{i} \rangle \mbox{ i.o.}\right)=0.
\] when $r \ge 4$ is assumed\footnote{``i.o." is an abbreviation of ``infinitely often".}.

	It remains to deal with the cases $r = 2, 3$. Recall that $v_{1}, \ldots v_{r}$ freely generate $\langle \supp \nu \rangle$. Let $Ab = (Ab_{1}, \ldots, Ab_{r}) : \langle \supp \nu \rangle \rightarrow \Z^{r}$ be the abelianization map sending each $v_{i}$ to $\vec{e}_{i}$. For $k = 0, 1, 2, 3$, we define \[ \begin{aligned}
	A_{k} := \{ w \in \langle \supp \nu \rangle : Ab_{1}(w) + k Ab_{2}(w) \in 4\Z \}, \\ B_{k} := \{ w \in \langle \supp \nu \rangle : kAb_{1}(w) + Ab_{2}(w) \in 4\Z \}.\end{aligned}
	\]
	Note that $\left( \cup_{k} A_{k} \right) \cup \left(\cup_{k} B_{k} \right) = \langle \supp \nu \rangle$. Further, $\langle \supp \nu \rangle / A_{k}$ and $\langle \supp \nu \rangle / B_{k}$ are isomorphic to $\Z / 4\Z$. Thus, $A_{k}$ and $B_{k}$ are index 4 subgroups of $\langle \supp \nu \rangle$; the Nielsen--Schreirer theorem asserts that $A_{k}$, $B_{k}$ are of rank $4r-3$. See Figure \ref{fig:coverA0} for the example $A_0$.
	
		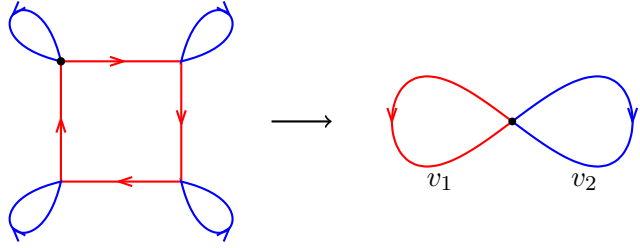
\begin{figure}[h]
		\centering
		\def\n{2.6}
		\def\m{0.9}
		\begin{tikzpicture}[scale=0.8, every node/.style={scale=1}]

		\begin{scope}[shift={(1.5, 0)}]
		\draw[thick, red,  decoration={markings, mark=at position 0.13 with {\draw (-0.18, 0.06) -- (0, 0) -- (-0.18, -0.06);}, mark=at position 0.38 with {\draw (-0.18, 0.06) -- (0, 0) -- (-0.18, -0.06);;}, mark=at position 0.63 with {\draw (-0.18, 0.06) -- (0, 0) -- (-0.18, -0.06);;},mark=at position 0.88 with {\draw (-0.18, 0.06) -- (0, 0) -- (-0.18, -0.06);}}, postaction={decorate}] (-1, -1) -- (-1, 1) -- (1, 1) -- (1, -1) -- cycle;
		\draw[thick, blue, decoration={markings, mark=at position 0.505 with {\draw (-0.18, 0.06) -- (0, 0) -- (-0.18, -0.06);}}, postaction={decorate}] (1, 1) .. controls (1.4, 2.7) and (2.7, 1.4) .. cycle;
		\draw[thick, blue, decoration={markings, mark=at position 0.505 with {\draw (-0.18, 0.06) -- (0, 0) -- (-0.18, -0.06);}}, postaction={decorate}]  (1, -1) .. controls (1.4, -2.7) and (2.7, -1.4) .. cycle;
		\draw[thick, blue, decoration={markings, mark=at position 0.505 with {\draw (-0.18, 0.06) -- (0, 0) -- (-0.18, -0.06);}}, postaction={decorate}]  (-1, -1) .. controls (-1.4, -2.7) and (-2.7, -1.4) .. cycle;
		\draw[thick, blue, decoration={markings, mark=at position 0.505 with {\draw (-0.18, 0.06) -- (0, 0) -- (-0.18, -0.06);}}, postaction={decorate}]  (-1, 1) .. controls (-1.4, 2.7) and (-2.7, 1.4) .. cycle;
		\fill (-1, 1) circle (0.07);
		\end{scope}
		
		\draw[thick, ->] (4, 0) -- (5, 0);
		
		\draw[thick, blue, domain=0:180,smooth,variable=\x, decoration={markings, mark=at position 0.5 with {\draw (-0.18, 0.06) -- (0, 0) -- (-0.18, -0.06);;}}, postaction={decorate}] plot ({8+2*sin(\x)},{1.5*sin(\x)*cos(\x)});
		\draw[thick, red, domain=180:360,smooth,variable=\x, decoration={markings, mark=at position 0.5 with {\draw (-0.18, 0.06) -- (0, 0) -- (-0.18, -0.06);;}}, postaction={decorate}] plot ({8+2*sin(\x)},{1.5*sin(\x)*cos(\x)});
		\fill (8, 0) circle (0.065);
		\draw (6.8, -1) node {$v_{1}$};
		\draw (9.2, -1) node {$v_{2}$};
		
		\end{tikzpicture}
		\caption{Covering space for $A_{0}$ in the case $r = 2$.}
		\label{fig:coverA0}
	\end{figure}

	Let us focus on $A_{0}$ from now on. Note that $v_{1} \in \langle \supp \nu \rangle$ corresponds to a generator of $\langle \supp \nu \rangle / A_{0} \cong \Z / 4\Z$. If each element of $\supp \nu$ corresponds to one of $\bar{0}, \bar{2} \in \Z/4\Z$, so does each element of $\langle \supp \nu \rangle$. This contradicts the existence of $v_{1}$; thus, some element $f$ of $\supp \nu$ should correspond to a generator of $\Z/4\Z$. In particular, $f^4 \in \llangle \supp \nu \rrangle \cap A_{0}$ holds\footnote{$\llangle \cdot \rrangle$ denotes the semigroup generated by $\{\cdot\}$.}.
	
	We define a stopping time $\{n_{m}\}_{m \in \Z_{\ge 0}}$ as \[
	n_{0} = 0, \quad n_{m} = \inf\{n>n_{m-1} : \omega_{n} \in A_{0}\}.
	\]
	Then $(\w'_{m}) = (\omega_{n_{m}})$ becomes a new random walk on $A_{0}$ with transition probability $\nu'$. Note that $\llangle \supp \nu' \rrangle = \llangle \supp \nu \rrangle \cap A_{0}$. Note also that $\w_{n}$ visits the conjugates of powers of $T_{A}$ \textit{inside $A_{0}$} i.o. if and only if $\w'_{m}$ visits the conjugates of powers of $T_{A}$ i.o..

	Let $\{v'_{1}, \ldots, v'_{4r-3}\}$ be a set of free generators of $A_{0}$. By \cite{hall1949coset} and \cite{burns1969free}, there exist $v'_{4r-2}, \ldots, v'_{k}\in \langle T_A, T_B \rangle$ so that $v'_{1}, \ldots, v'_{k}$ freely generate a finite index subgroup $L$ of $\langle T_A, T_B \rangle$. Let $\{L g_{1}, \ldots, Lg_{l}\}$ be the collection of the right cosets of $L$. 
	
	As before, we define the abelianization map $Ab' : L \rightarrow \Z^{k}$ sending each $v_{i}'$ to $\vec{e}_{i}$ and a projection map $Q : \Z^{k} \rightarrow \Z^{4r-3}$ discarding the last $k-~4r+3$ coordinates. For $i = 1, \ldots, l$, we pick  $w_{i} \in \Z^{4r-3} \setminus \{0\}$ such that $(Q \circ~Ab')(\langle g_{i} T_{A} g_{i}^{-1} \rangle \cap L) \subseteq \langle w_{i}\rangle$. We then similarly deduce that \[
	\Prob \left( \begin{matrix} \w'_m \mbox{ visits the conjugates of} \\ \mbox{powers of } T_A \mbox{ i.o.} \end{matrix} \right) \le \sum_{i=1}^{l} \Prob\left((Q \circ Ab')( \w'_{m}) \in \langle w_{i} \rangle \mbox{ i.o.}\right).
	\]
	
	For each $i$, we consider the projection $P_{i} : \R^{4r-3} \rightarrow \R^{4r-4}$ onto the orthogonal complement of $w_{i}$. Suppose that $(P_{i} \circ Q \circ Ab')(\supp \nu') \cdot \theta = 0$ for some $\theta \in \R^{4r-4} \setminus \{0\}$. In particular, $f^4 \in \llangle \supp \nu \rrangle \cap A_{0} = \llangle \supp \nu' \rrangle$ also satisfies $(P_{i} \circ Q \circ Ab')(f^4) \cdot \theta = 0$. 
	
	Since $(P_{i} \circ Q \circ Ab')(\langle \supp \nu \rangle\cap A_{0}) = P_{i}(\Z^{r}) \cdot \theta \neq 0$, $(P_{i} \circ Q \circ Ab') (w) \cdot \theta \neq 0$ for some $w \in \langle \supp \nu \rangle \cap A_{0}$. Here, $w = b_{1}^{k_{1}} \cdots b_{m}^{k_{m}}$ for some  alphabets $b_{j} \in \supp \nu$ and $k_{j} \in \Z$. Using $f$, we express $w$ as a concatenation of elements in $A_{0}$: \[
	w = (b_{1}^{k_{1}} f^{-t_{1}} ) (f^{t_{1}} b_{2}^{k_{2}}f^{-t_{2}}) \cdots (f^{t_{m-1}} b_{m}^{k_{m}}),
	\]
	where $f^{t_{j}}$ and $b_{1}^{k_{1}} \cdots b_{j}^{k_{j}}$ belong to the same right coset of $A_{0}$ in $\langle \supp \nu \rangle$.
	
	Since $(P_{i} \circ Q \circ Ab')(w) \cdot \theta \neq 0$, we deduce that $(P_{i} \circ Q \circ Ab')(f^{t} b^{k} f^{s}) \cdot \theta \neq 0$ for some $b \in \supp \nu$ and $t, k, s \in \Z$ such that $f^tb^kf^s \in A_0$. By taking inverse if necessary, we may suppose $k \ge 0$. Recall again that $f^4 \in A_0$ and $(P_{i} \circ Q \circ Ab')(f^4) \cdot \theta = 0$; by multiplying $f^{-4t}$ or $f^{-4s}$ if necessary, we may also assume that $t, s \ge 0$. However, $f^{t} b^{k} f^{s}$ then belongs to $\llangle \supp \nu \rrangle \cap A_{0} = \llangle \supp \nu' \rrangle$, which contradicts the assumption. 
	
	Thus, $(P_{i} \circ Q \circ Ab')(\supp \nu') \cdot \theta \neq 0$ for each $\theta \in \R^{4r-4} \setminus \{0\}$; $P_{i} \circ Q \circ Ab' \circ \w'$ is truly $(4r-4)$-dimensional and thus transient. Consequently, we deduce \[
\Prob \left( \begin{matrix} \w_{n} \mbox{ visits the conjugates of} \\ \mbox{powers of } T_A \mbox{ in } A_{0} \mbox{ i.o.} \end{matrix} \right) \le \sum_{i=1}^{l} \Prob\left((Q \circ Ab')(\w'_{m}) \in \langle w_{i} \rangle \mbox{ i.o.}\right)=0.
\]

Similarly, almost every $\w_{n}$ eventually avoids the conjugates of powers of $T_{A}$ in each of $A_{k}$ and $B_{k}$. Put together, almost every $\w_{n}$ eventually avoids any conjugates of powers of $T_{A}$. Similar argument applies to the conjugates of powers of $T_{B}$ or $T_{A}T_{B}$. Then Theorem \ref{thm:classofelts} completes the proof.
\end{proof}

%
%

\medskip
\section{Applications}	\label{sec:appmat}

\subsection{Volumes of random mapping tori from Thurston's construction}

As we explained earlier, the topological entropy of a random walk is closely related to the dynamics of random mapping classes. In this section, we discuss applications of the eventually pseudo-Anosov behavior of random walks and the spectral theorem. As the first application, we estimate the hyperbolic volumes of random mapping tori with monodromies obtained from Thurston's construction.

\begin{definition}[Mapping torus]
	For a surface $S$ and a mapping class $\varphi \in \Mod(S)$, the \emph{mapping torus $M_{\varphi}$ with monodromy $\varphi$} is the closed 3-manifold definde by $$M_{\varphi} := S \times [0, 1] / (x, 0) \sim (\varphi(x), 1).$$
\end{definition}

Thurston proved that the geometry of a mapping torus $M_{\varphi}$ reads off the nature of the mapping class $\varphi$. For instance, according to \cite{otal1996theoreme}, $\varphi \in~\Mod(S)$ is pseudo-Anosov if and only if $M_{\varphi}$ admits a hyperbolic structure, which is unique by the Mostow rigidity theorem \cite{mostow1968quasi}. In this case, the \emph{hyperbolic volume} $\vol(M_{\varphi})$ of $M_{\varphi}$ serves as an invariant of the mapping class. 

Once we have a probability measure $\nu$ on $\Mod(S)$ and the induced one $\Prob$ on $\Mod(S)^{\N}$, we get a sequence of random mapping tori $M_{\w_n}$ corresponding to a random walk $\w = (\w_n) \in \Mod(S)^{\N}$. For an eventually pseudo-Anosov sample path $\w_{n}$, it makes sense to consider the asymptotes of the hyperbolic volume of $M_{\w_{n}}$. For instance, Viaggi proved the following.

\begin{thm}[Viaggi, \cite{viaggi2019volumes}] \label{thm:viaggi}
	Suppose that the probability measure $\nu$ on $\Mod(S)$ satisfies $\langle \supp \nu \rangle = \Mod(S)$ and it has a finite symmetric support. Then for the induced probability measure $\Prob$ on $\Mod(S)^{\N}$, $$\lim_{n \to \infty} {1 \over n } \vol(M_{\w_n}) = V\quad \mbox{almost surely}$$ for some $V > 0$.
\end{thm}

According to Viaggi \cite{viaggi2019volumes}, the condition in Theorem \ref{thm:viaggi} can be weakened so that $\nu$ is non-elementary, i.e., $\langle \supp \nu \rangle$ is not necessarily the whole $\Mod(S)$ but a non-elementary subgroup. However, Viaggi's argument relies on the boundedness of $\supp \nu$, that is, $\sup_{g \in \supp \nu} d_{\T}(\X, g \cdot \X) < \infty$ for some $\X \in \T(S)$, in order to estimate $\vol(M_{\w_n})$ by applying triangle inequalities. Therefore, one cannot affirm that ${1 \over n} \vol(M_{\w_n})$ converges if the probability measure is not finitely supported.

We now consider this question in the setting of Thurston's construction. As before, let $(A, B)$ be a filling pair of multicurves such that $\langle T_A, T_B \rangle \cong F_{2}$ and $\nu$ be a non-elementary probability measure on $\langle T_A, T_B \rangle$. This induces a measure $\Prob$ on the path space $\Mod(S)^{\N}$.  If $\rk \langle \supp \nu \rangle < \infty$, then Theorem \hyperref[thm:eventuallypA]{B} asserts that $\Prob-a.e.$ random walk $\w$ has a pseudo-Anosov tail. In this circumstance, we can discuss the long-term behavior of $\vol(M_{\w_n})$.

Regarding the hyperbolic volume of a mapping torus, Kojima and McShane provided the following estimation in terms of the topological entropy of the monodromy. We note that the work of Brock \cite{brock2003weil} and Linch \cite{linch1974comparison} leads to a similar estimation with a different constant. 

\begin{thm}[Kojima--McShane, \cite{kojima2018normalized}]\label{thm:kojima}
	For any pseudo-Anosov $\varphi \in \Mod(S)$ and its topological entropy $h(\varphi) = \log \lambda_{\varphi}$, we have $$\vol(M_{\varphi}) \le -3\pi\chi(S) \cdot h(\varphi)$$ where $\chi(S)$ is the Euler characteristic of $S$.
\end{thm}

Combining this estimation with Theorem \hyperref[thm:eventuallypA]{B} gives the following statement.

\mappingtori*

\begin{proof}

As we mentioned above, Theorem \hyperref[thm:eventuallypA]{B} ensures that $\w_{n}$ is eventually pseudo-Anosov for $\Prob-a.e.$ $\w$. By Theorem \ref{thm:kojima}, we deduce $$\vol(M_{\w_n}) \le -3 \pi \chi(S) \cdot \log \lambda_{\w_n}$$ for large $n$, where $\lambda_{\w_n}$ is the stretch factor of $\w_n$. By Theorem \ref{thm:bers}, $${1 \over n} \vol(M_{\w_n}) \le -3 \pi \chi(S) \cdot {1 \over n } \log \lambda_{\w_n} \le -3\pi \chi(S) \cdot {1 \over n} d_{\T}(\mathcal{Y}, \w_n \cdot \mathcal{Y})$$ for any $\mathcal{Y} \in \T(S)$. Then applying Theorem \ref{thm:masai} completes the proof.
\end{proof}

\begin{remark} \label{rem:kojimamcshane}
	This estimation suggests a similarity between the topological entropy of a random walk and that of a single pseudo-Anosov. Indeed, $$\vol(M_{\varphi^n}) = n \cdot \vol(M_{\varphi}), \quad h(\varphi^n) = n\cdot h(\varphi)$$ and thus $$\limsup_{n \to \infty} {1 \over n} \vol(M_{\varphi^n}) \le -3\pi \chi(S) \cdot h(\varphi)$$ for any pseudo-Anosov $\varphi \in \Mod(S)$.
	
\end{remark}

\subsection{The distribution of the stretch factors from Thurston's construction} \label{subsec:distribution}

We now discuss the distribution of the stretch factors. Note that for a non-elementary probability measure $\nu$ on $\langle T_A, T_B \rangle$ with finite first moment, Theorem \ref{thm:inprobconv} implies that for each $\varepsilon > 0$, \[
\Prob\left(e^{h(\w)- \varepsilon} < \lambda_{\w_{n}} < e^{h(\w) + \varepsilon}\right) \rightarrow 1\quad \mbox{as } n \rightarrow \infty.
\]

We now discuss the case of finitely supported measures in detail. Let $\nu$ be a finitely supported non-elementary measure on $\langle T_A, T_B \rangle$ and denote $p := \min \{ \nu(w) : w \in \supp \nu \} > 0$. Let $\X_{0} \in \T(S)$ be the point corresponding to the origin of the Poincar\'e disk $\mathbb{H}^{2}$. Note that $$\bigcup_{N = 0} ^{\infty} \{\w_N : \w = (\w_n) \in \Omega \} = \llangle \supp \nu \rrangle $$ for any $\Omega \subseteq \Mod(S)^{\N}$ with $\Prob(\Omega) = 1$. Hence, observing $\Prob-a.e. \w$ is sufficient to deal with the stretch factors of all pseudo-Anosovs in $\llangle \supp \nu \rrangle$.

Since $\nu$ is finitely supported in addition, even stronger estimate is available. Using the estimate of Azuma \cite{azuma1967sum} for martingales with bounded increments, the theory of Benoist and Quint in \cite{benoist2016linear} yields the following: there exist constants $C$ and $\sigma$ such that for any $\varepsilon > 0$ and $x \in \partial_B M$, 
\begin{align}\label{eqn:firstExpDecay}\Prob \left(|h_x(\w_n^{-1}\cdot \X_{0}) - n L_{\T} | > \varepsilon n\right) \le C e^{-(n\varepsilon^{2})/2\sigma^{2}} \mbox{ and} 
\\ \label{eqn:secondExpDecay}\Prob \left(|d_{\mathbb{H}^2}(\X_{0}, \w_n \cdot \X_{0}) - n L_{\T}| > \varepsilon n \right) \le C e^{-(n\varepsilon^{2})/2\sigma^{2}}.
\end{align}
Here $L_{\T}$ is the drift of the random walk. Taking $\varepsilon > 4\sigma\sqrt{- \log p}$, the proof of Theorem \hyperref[thm:secondmoment]{A} asserts that \[
\Prob\left( |\log \lambda_{\w_{n}} - n L_{\T}| > 6 \varepsilon n \right) < p^{n}
\]
for sufficiently large $n$. Since all atoms at step $n$ has probability at least $p^{n}$, this implies that $n(L_{\T} - 6\varepsilon) < \log \lambda_{\w_{n}} < n (L_{\T} + 6 \varepsilon)$ for sufficiently large $n$. In other words, we conclude the following.

\distribution* \label{thm:estimation}

It remains to clarify $\sigma$ in Inequality \ref{eqn:firstExpDecay} and Inequality \ref{eqn:secondExpDecay}. For example, we can choose $\sigma$ to be $M_{1} := \max \{ d_{\mathbb{H}^2}(\X_{0}, w \X_{0}) : w \in \supp \nu \}$. Since $p \le 1/2$, we obtain $n (L_{\T} + 6 \varepsilon) > (24 \sqrt{ \log 2}) n M_{1}$. This estimate is no better than $n M_{1}$, which is obtained as follows. Let $w =a_{1} \cdots a_{n}$ be a pseudo-Anosov with $a_{i} \in\supp \nu$, and let $w_{i} = a_{1} \cdots a_{i}$ for $i=1, \ldots, n$. We then observe  \[
\log \lambda_{w} \le d_{\mathbb{H}^2}(\X_{0}, w \cdot \X_{0}) \le d_{\mathbb{H}^2}(\X_{0}, w_{1} \cdot \X_{0}) +  \sum_{i=2}^{n} d_{\mathbb{H}^2}( w_{i-1} \cdot \X_{0}, w_{i} \cdot \X_{0}) \le n M_{1}.
\]
Note also that this estimate applies to every step $n$.

However, $M_{1}$ may be exaggerated in some cases, whereas the drift $L_{\T}$ reads off the cancellation among elements and provides a better bound for the stretch factor. For example, we can fix $s \in \langle T_A, T_B \rangle$ and set $\nu(s T_{A}^{\pm 1} s^{-1}) = \nu (s T_{B}^{\pm 1} s^{-1}) = 1/4$. Then there exists a constant $C > 0$ such that $d_{\mathbb{H}^2}(\X_{0}, a \cdot~\X_{0}) \ge 2 d_{\mathbb{H}^2}(\X_{0}, s \cdot \X_{0}) - C$ for each $a \in \supp \nu$. Thus, both $M_{1}$ and the first moment is at least $2 d_{\mathbb{H}^2}(\X_{0}, s\cdot \X_{0}) - C$.

Now let $w$ be a word with $k$ alphabets in $\supp \nu$. Then we obtain \[
d_{\mathbb{H}^2}(\X_{0}, w \cdot \X_{0}) \le 2 d_{\mathbb{H}^2}(\X_{0}, s\cdot \X_{0}) + k K
\] where \[
K = d_{\mathbb{H}^2}(\X_0, T_{A}^{\pm 1}\cdot \X_0) = d_{\mathbb{H}^2}(\X_0, T_{B}^{\pm 1} \cdot \X_0) = \log {\sqrt{\mu} + \sqrt{4 + \mu} \over 2}.
\]
Thus, for any $K' > K$ and sufficiently large $k$, \[
\max \{d_{\mathbb{H}^2}(\X_{0}, w\cdot \X_{0}) : w =a_{1} \cdots a_{k}, \ a_{i} \in \supp \nu\} < kK'.
\]
In this situation, Inequality \ref{eqn:firstExpDecay} and Inequality \ref{eqn:secondExpDecay} with $\sigma = K'$ and $\varepsilon > 4 K' \sqrt{ \log 4}$ hold for sufficiently large $n$. Note also $L_{\T} \le K'$. Thus, when $d_{\mathbb{H}^{2}}(\X_0, s \cdot \X_0) > 100K' + C$, $n(L_{\T} + 6\varepsilon)$ becomes a better bound than $n M_{1}$. 

Here, using the drift amounts to using $\X =  s \cdot \X_{0}$. With this different pivot point, one can also obtain the bound $\log \lambda_{\w_{n}} \le nK$ for all $n$.

The lower bound $n(L_{\T} - 6\varepsilon)$ may not be informative when $\varepsilon > L_{\T}/6$. This happens, for example, if $\supp \nu$ is symmetric. Nonetheless, Proposition \ref{prop:lowerboundtranslationlength} gives another lower bound for $\lambda_{\w_{n}}$ in terms of the cyclically reduced word norm of $\w_{n}$. Combining them yields the following theorem.

\begin{thm} \label{thm:constantestimate}
	Let $(A, B)$ be a filling pair of multicurves on $S$. Then $$\lambda_{\varphi} \le e^{K|\varphi|}$$ for any pseudo-Anosov $\varphi \in \langle T_A, T_B \rangle$ where $| \cdot |$ is the cyclically reduced word norm with respect to $T_A^{\pm}$ and $T_B^{\pm}$ and $K = \log {\sqrt{\mu} + \sqrt{4 + \mu} \over 2}$.
	
	Moreover, if $\langle T_A, T_B \rangle \cong F_2$, each pseudo-Anosov $\varphi \in \langle T_A, T_B \rangle$ satisfies  $$|\varphi|e^{1/4} \le \lambda_{\varphi} \le e^{K|\varphi|}.$$
\end{thm}

\begin{remark}
	This estimation is \emph{non-probabilistic}.
\end{remark}

It is also interesting to understand the algebraic properties of the stretch factors of pseudo-Anosovs obtained from Thurston's construction. Many known examples are Salem numbers, which are defined as follows.

\begin{definition}[Salem number]
	A real algebraic unit $\lambda > 1$ is a \emph{Salem number} if all Galois conjugates except $\lambda^{\pm 1}$ lie on the unit circle in $\mathbb{C}$.
\end{definition}

One can naturally ask whether every Salem number can be realized as the stretch factor of some pseudo-Anosov from Thurston's construction, which is a special case of Fried's conjecture. Pankau partially answered this question by proving that every Salem number has a power which is the stretch factor arising from Thurston's construction.

\begin{thm}[Pankau, \cite{pankau2017salem}]
	For any Salem number $\lambda$, there are $k, g \in \N$ such that $\lambda^k$ is the stretch factor of a pseudo-Anosov homeomorphism of the surface $S_{g}$ genus $g$ arising from Thurston's construction.
\end{thm}

In fact, Pankau's proof concludes that the pseudo-Anosov for a given Salem number can be obtained as a word of two positive multitwists $T_A$ and $T_B$. Together with Theorem \ref{thm:constantestimate}, we get the following corollary.

\begin{cor}
	Let $\lambda$ be a Salem number and suppose that $\lambda^k$ is the stretch factor of a pseudo-Anosov $\varphi$ from Thurston's construction. Then we have $$\lambda < \lambda^k \le e^{K |\varphi|}.$$
\end{cor}

\begin{proof}
	The existence of $k$ and a configuration for Thurston's construction is guaranteed by the work of Pankau. Then by Theorem \ref{thm:constantestimate}, we conclude that $\lambda^k = \lambda_{\varphi} \le e^{K |\varphi|}$.
\end{proof}

One interesting consequence of above inequality, together with Proposition \ref{prop:lowerboundtranslationlength}, is that we can estimate the power $k$.

\salempower* \label{thm:salempower}

In fact, Pankau constructed a pseudo-Anosov of word norm two whose stretch factor is some power of a given Salem number $\lambda$. Accordingly, one can rewrite the above inequalities if one focuses only on the specific configuration for Thurston's construction considered by Pankau in \cite{pankau2017salem}.

%
%

\medskip
\bibliographystyle{alpha}
\bibliography{thurston}

\end{document}